\documentclass[12pt]{article}
\usepackage{amsmath}
\usepackage{graphicx}
\usepackage{enumerate}
\usepackage{natbib}
\usepackage{url} 
\usepackage{amsthm}
\usepackage[colorlinks,citecolor=blue,urlcolor=blue]{hyperref}
\usepackage{geometry}
\usepackage{paralist}
\usepackage{subfigure}
\usepackage{amssymb}
\usepackage{amsfonts}
\usepackage{color}
\usepackage{multirow}
\usepackage[linesnumbered,ruled,vlined]{algorithm2e}
\usepackage{wrapfig,lipsum,booktabs}
\usepackage{enumitem}  

\graphicspath{
{./fig/}
}

\makeatletter
\newcommand{\pushright}[1]{\ifmeasuring@#1\else\omit\hfill$\displaystyle#1$\fi\ignorespaces}
\newcommand{\pushleft}[1]{\ifmeasuring@#1\else\omit$\displaystyle#1$\hfill\fi\ignorespaces}
\makeatother

\newcommand{\blind}{1}

\newtheoremstyle{mytheorem}{}{}{\slshape}{}{\bfseries}{}{2mm}{}
\newtheoremstyle{mytheorem2}{}{}{}{}{\bfseries}{}{2mm}{}
\theoremstyle{mytheorem}
\newtheorem{theorem}{Theorem}
\newtheorem{lemma}{Lemma}

\newtheorem{proposition}{Proposition}
\theoremstyle{mytheorem2}
\newtheorem{remark}{Remark}

\newtheorem{definition}{Definition}
\newcommand{\iid}{\stackrel{\text{iid}}{\sim}}
\newcommand{\ind}{\stackrel{\text{ind}}{\sim}}
\DeclareMathOperator{\ve}{vec}
\DeclareMathOperator{\tr}{tr}
\DeclareMathOperator{\var}{Var}
\DeclareMathOperator{\diag}{diag}

\DeclareMathOperator*{\argmin}{arg\,min}
\newcommand{\bd}{\boldsymbol} 

\begin{document}

\def\spacingset#1{\renewcommand{\baselinestretch}%
{#1}\small\normalsize} \spacingset{1}


\if1\blind
{
  \title{\bf An Empirical Bayes Approach to Shrinkage Estimation on the Manifold
  of Symmetric Positive-Definite Matrices\thanks{This research was in part
  funded by the NSF grants IIS-1525431 and IIS-1724174 to Vemuri.} \footnote{This
  manuscript has been submitted to a journal.}}
  
  \author{Chun-Hao Yang$^1$, Hani Doss$^1$ and Baba C. Vemuri$^2$\\
    $^1$Department of Statistics\\
    $^2$Department of Computer Information Science \& Engineering\\ University of Florida}
  \maketitle
} \fi

\if0\blind
{
  \bigskip
  \bigskip
  \bigskip
  \begin{center}
    {\LARGE\bf  An Empirical Bayes Approach to Shrinkage Estimation on the
    Manifold of Symmetric Positive-Definite Matrices}
\end{center}
  \medskip
} \fi

\bigskip
\begin{abstract}
The James-Stein estimator is an estimator of the multivariate normal mean and dominates the maximum likelihood estimator (MLE) under squared error loss. The original work inspired great interest in developing shrinkage estimators
for a variety of problems. Nonetheless, research on shrinkage estimation for manifold-valued data is scarce. In this paper, we propose shrinkage estimators for the parameters of the Log-Normal distribution defined on the manifold of $N
\times N$ symmetric positive-definite matrices. For this manifold, we choose the Log-Euclidean metric as its Riemannian metric since it is easy to compute and is widely used in applications. By using the Log-Euclidean distance in the loss
function, we derive a shrinkage estimator in an analytic form and show that it is asymptotically optimal within a large class of estimators including the MLE, which is the sample Fr\'{e}chet mean of the data. We demonstrate the performance of the proposed shrinkage estimator via several simulated data experiments.
Furthermore, we apply the shrinkage estimator to perform statistical inference in diffusion magnetic resonance imaging problems. 

\end{abstract}

\noindent%
\vfill

\newpage
\spacingset{1.5} 
\section{Introduction}\label{intro}

Symmetric positive-definite (SPD) matrices are common in applications of
science and engineering. In computer vision problems, they are
encountered in the form of covariance matrices, e.g.\ region covariance descriptor \citep{tuzel2006region}, and in diffusion magnetic
resonance imaging, SPD matrices manifest themselves as diffusion tensors which
are used to model the diffusion of water molecules \citep{basser1994mr} and as Cauchy deformation
tensors in morphometry to model the deformations (see \citet[Ch.\ 36]{frackowiak2004human}). Many other applications can be found in \citet{cherian2016positive}. In such applications, the
statistical analysis of data must perform geometry-aware computations, i.e.\
employ methods that take into account the nonlinear geometry of the data space.
In most data analysis applications, it is useful to describe the entire dataset
with a few summary statistics. For data residing in Euclidean space, this may 
simply be the sample mean, and for data residing in non-Euclidean spaces,
e.g.\ Riemannian manifolds, the corresponding statistic is the sample
Fr\'{e}chet mean (FM) \citep{frechet1948}. The sample FM also plays an important
role in different statistical inference methods, e.g.\ principal geodesic
analysis \citep{fletcher2003gaussian}, clustering algorithms, etc. If $M$ is a
metric space with metric $d$, and $x_1,\ldots, x_n \in M$, the sample FM is
defined by $\bar{x}=\argmin_{m}\sum_{i=1}^n d^2(x_i, m)$. For Riemannian
manifolds, the distance is usually chosen to be the intrinsic distance induced
by the Riemannian metric. Then, the above optimization problem can be solved by
Riemannian gradient descent algorithms
\citep{pennec2006intrinsic, groisser2004,afsari2011,moakher2005differential}. However, Riemannian
gradient descent algorithms are usually computationally expensive, and efficient
recursive algorithms for computing the sample FM have been presented in the literature for various
Riemannian manifolds by \cite{Sturm03}, \cite{Cheng-AISTATS13},
\cite{salehian2015efficient}, \cite{chakraborty2015recursive},
\cite{lim2014weighted} and \cite{chakraborty2019statistics}. 

In $\mathbb{R}^p$ with the Euclidean metric, the sample FM is just the ordinary
sample mean. Consider a set of normally distributed random variables
$X_1,\ldots,X_n$. The sample mean $\bar{X}=n^{-1}\sum_{i=1}^n X_i$ is the
maximum likelihood estimator (MLE) for the mean of the underlying normal
distribution, and the James-Stein (shrinkage) estimator \citep{james1961estimation} was shown
to be better (under squared error loss) than the MLE when $p>2$ and the
covariance matrix of the underlying normal distribution is assumed to be known.
Inspired by this result, the goal of this paper is to develop shrinkage
estimators for data residing in $P_N$, the space of $N \times N$ SPD matrices.

For the model $X_i \ind N(\mu_i, \sigma^2)$ where $p > 2$ and $\sigma^2$ is
known, the MLE of $\mu=[\mu_1, \ldots, \mu_p]^T$ is $\hat{\mu}^{\text{MLE}} =
[X_1,\ldots,X_p]^T$ and it is natural to ask whether the MLE is admissible.
\cite{stein1956} gave a negative answer to this question and provided a class of
estimators for $\mu$ that dominate the MLE. Subsequently,
\cite{james1961estimation} proposed the estimator
\begin{equation} \label{eqn:JS}
	\left( 1 - \frac{(p-2)\sigma^2}{\left\| X \right\| ^2} \right)X
\end{equation}
where $X = \left[X_1,\ldots,X_p\right]^T$, which was later referred to as the
James-Stein (shrinkage) estimator.

Ever since the work reported in \cite{james1961estimation}, generalizations or
variants of this shrinkage estimator have been developed in a variety of
settings. A few directions for generalization are as follows. First, instead of
estimating the means of many normal distributions, carry out the simultaneous
estimation of parameters (e.g.\ the mean or the scale parameter) of other
distributions, e.g.\ Poisson, Gamma, etc.; see \cite{mck1971admissible},
\cite{fienberg1973simultaneous}, \cite{clevenson1975simultaneous},
\cite{tsui1981simultaneous}, \cite{tsui1982simultaneous},
\cite{brandwein1990stein} and \cite{brandwein1991generalizations}. More recent
works include \cite{xie2012sure}, \cite{xie2016optimal}, \cite{jing2016sure},
\cite{kong2017sure}, \cite{muandet2016kernel}, \cite{hansen2016efficient}, and \cite{feldman2014revisiting}. Second, since the MLE of the mean of a normal distribution
is also the best translation-equivariant estimator but is inadmissible, an
interesting question to ask is, when is the best translation-equivariant
estimator admissible? This problem was studied extensively in
\cite{stein1959admissibility} and \cite{brown1966admissibility}. 

Besides estimation of the mean of different distributions, estimation of the
covariance matrix (or the precision matrix) of a multivariate normal
distribution is an important problem in statistics, finance, engineering and
many other fields. The usual estimator, namely the sample covariance matrix,
performs poorly in high-dimensional problems and many researchers have endeavored to
improve covariance estimation by applying the concept of shrinkage in this
context \citep{stein1975, haff1991variational, daniels2001shrinkage, 
ledoit2003improved, daniels2001shrinkage, ledoit2012nonlinear, donoho2018optimal}.
Note that there is a vast literature on covariance estimation and we only cited a
few references here. For a thorough literature review, we refer the interested
reader to \cite{donoho2018optimal}.

In order to understand shrinkage estimation fully, one must understand why the
process of shrinkage improves estimation. In this context, Efron and Morris
presented a series of works to provide an empirical Bayes interpretation by
modifying the original James-Stein estimator to suit different problems
\citep{efron1971limiting, efron1972limiting, efron1972empirical, efron1973stein,
efron1973combining}. The empirical Bayes approach to designing a shrinkage
estimator can be described as follows. First, reformulate the model as a
Bayesian model, i.e.\ place a prior on the parameters. Then, the
hyper-parameters of the prior are estimated from the data. \cite{efron1973stein}
presented several examples of different shrinkage estimators developed within
this empirical Bayes framework. An alternative, non-Bayesian, geometric
interpretation to Stein shrinkage estimation was presented by
\cite{brown2012geometrical}.

In all the works cited above, the domain of the data has invariably been a
vector space and, as mentioned earlier, many applications naturally encounter
data residing in non-Euclidean spaces. Hence, generalizing shrinkage estimation
to non-Euclidean spaces is a worthwhile pursuit. In this paper, we focus on
shrinkage estimation for the Riemannian manifold $P_N$. We assume that the
observed SPD matrices are drawn from a Log-Normal distribution defined on $P_N$
\citep{schwartzman2016lognormal} and we are interested in estimating the mean
and the covariance matrix of this distribution. We derive shrinkage estimators
for the parameters of the Log-Normal distribution using an empirical Bayes
framework \citep{xie2012sure}, which is described in detail subsequently, and
show that the proposed estimator is asymptotically optimal within a class of
estimators including the MLE. We describe simulated data experiments which
demonstrate that the proposed shrinkage estimator of the mean of the Log-Normal
distribution is better (in terms of risk) than the sample FM, which is the MLE, and the shrinkage estimator proposed by \citet{yang2019shrinkage}. Further, we also apply the
shrinkage estimator to find group differences between patients with Parkinson's
disease and controls (normal subjects) from their respective brain scans
acquired using diffusion magnetic resonance images (dMRIs). 

The rest of this paper is organized as follows. In section~\ref{prelim}, we
present relevant material on the Riemannian geometry of $P_N$ and shrinkage
estimation. The main theoretical results are stated in section~\ref{theory} with
the proofs of the theorems relegated to the supplement. In section~\ref{exp}, we
demonstrate how the proposed shrinkage estimators perform via several synthetic
data examples and present applications to (real data) diffusion tensor imaging
(DTI), a clinically popular version of dMRI. Specifically, we apply 
the proposed shrinkage estimator to the estimation of
the brain atlases (templates) of patients with Parkinson's disease and a
control group and identify the regions of the brain where the two groups differ significantly.
Finally, in section~\ref{conc} we discuss our contributions and present some
future research directions.

\section{Preliminaries}\label{prelim}

In this section, we briefly review the commonly used Log-Euclidean metric for
$P_N$ \citep{arsigny2007geometric} and the concept of Stein's unbiased
risk estimate, which will form the framework for deriving the shrinkage
estimators. 

\subsection{Riemannian Geometry of \texorpdfstring{$P_N$}{Lg}}

In this work, we endow the manifold $P_N$ with the Log-Euclidean metric proposed by \citet{arsigny2007geometric}. We note that there is another commonly used Riemannian metric on $P_N$, called the affine-invariant metric (see \citet[Ch.\ 1]{terras2016harmonic} for its introduction and \citet{lenglet2006statistics} and \citet{moakher2005differential} for its applications). The affine-invariant metric is computationally more expensive and in some applications it provides results that are indistinguishable from those obtained under the Log-Euclidean metric as demonstrated in \citet{arsigny2007geometric} and \citet{schwartzman2016lognormal}. Based on this reasoning, we choose to work with the Log-Euclidean metric. For other metrics on $P_N$ used in a variety of applications, we refer the reader to a recent survey \citep{feragen2017geometries}. 

The Log-Euclidean metric is a bi-invariant Riemannian metric on the abelian Lie group $(P_N, \odot)$ where $X\odot Y = \exp(\log X + \log Y)$. The intrinsic distance $d_{\text{LE}}: P_N \times P_N\to\mathbb{R}$ induced by the Log-Euclidean metric has a very simple form, namely
\[
    d_{\text{LE}}(X,Y) = \left\Vert \log X - \log Y \right\Vert
\] 
where $\Vert \cdot \Vert$ is the Frobenius norm. Consider the map $\ve:\textsf{Sym}(N) \to \mathbb{R}^{\frac{N(N+1)}{2}}$ defined by 
\[
    \ve(Y) = \left[  y_{11},\ldots, y_{nn}, \sqrt{2}( y_{ij})_{i<j}\right]^T
\]
\citep{schwartzman2016lognormal}. This map is actually an isomorphism between
$\textsf{Sym}(N)$ and $\mathbb{R}^{\frac{N(N+1)}{2}}$. To make the notation more
concise, for $X \in P_N$, we denote $\widetilde{X} = \ve(\log X) \in
\mathbb{R}^{\frac{N(N+1)}{2}}$. From the definition of $\ve$, we see that
$d_{\text{LE}}(X, Y) = \Vert \widetilde{X} - \widetilde{Y} \Vert$.

Given $X_1,\ldots,X_n \in P_N$, we denote the sample FM with respect
to the two intrinsic distances given above by
\begin{align} 
\label{eqn:FM_LE}
\bar{X} & = \argmin_{M\in P_N} n^{-1} \sum_{i=1}^n d_{LE}^2(X_i, M) = \exp\left( n^{-1}\sum_{i=1}^n \log X_i \right).
\end{align}

\subsection{The Log-Normal Distribution on \texorpdfstring{$P_N$}{Lg}}

In this work, we assume that the observed SPD matrices follow the Log-Normal
distribution introduced by \cite{schwartzman2006random}, which can be viewed as a
generalization of the Log-Normal distribution on $\mathbb{R}^{+}$ to $P_N$. The
definition of the Log-Normal distribution is stated as follows.

\begin{definition} \label{def:lognormal}
Let $X$ be a $P_N$-valued random variable. We say $X$ follows a Log-Normal
distribution with mean $M \in P_N$ and covariance matrix $\Sigma \in
P_{N(N+1)/2}$, or $X \sim \text{LN}(M, \Sigma)$, if $\widetilde{X} \sim
N(\widetilde{M}, \Sigma)$.
\end{definition}

From the definition, it is easy to see that $E\log X = \log M$ and $E\Vert \! \log X
- \log M \Vert^2 = E\Vert \widetilde{X} - \widetilde{M}\Vert ^2 = \tr(\Sigma)$.
Some important results regarding this distribution were obtained in
\cite{schwartzman2016lognormal}. The following proposition for the MLEs of the
parameters will be useful subsequently in
this work.

\begin{proposition} \label{prop:MLE}
Let $X_1,\ldots,X_n \iid \text{LN}(M, \Sigma)$. Then, the MLEs of $M$ and
$\Sigma$ are $\widehat{M}^{\text{MLE}} = \bar{X}$ and
$\widehat{\Sigma}^{\text{MLE}} = n^{-1} \sum_{i=1}^n \Big( \widetilde{X}_i -
\widetilde{\widehat{M}^{\text{MLE}}}\Big)\Big( \widetilde{X}_i -
\widetilde{\widehat{M}^{\text{MLE}}}\Big)^T$. The MLE of $M$ is the sample FM under
the Log-Euclidean metric.
\end{proposition}

\subsection{Bayesian Formulation of Shrinkage Estimation in
\texorpdfstring{$\mathbb{R}^p$}{Lg}}
\label{subsec:bayes}

As discussed earlier, the James-Stein estimator originated from the problem of
simultaneous estimation of multiple means of (univariate) normal distributions.
The derivation relied heavily on the properties of the univariate normal
distribution. Later on, \cite{efron1973stein} gave an empirical Bayes
interpretation for the James-Stein estimator, which is presented by considering
the hierarchical model
\begin{eqnarray*}
	X_i|\theta_i & \ind & N(\theta_i, A),\quad i=1,\ldots,p,\\
	\theta_i & \iid & N(\mu, \lambda),
\end{eqnarray*}
where $A$ is known and $\mu$ and $\lambda$ are unknown. The posterior mean for
$\theta_i$ is
\begin{equation} \label{eqn:map}
\hat{\theta}^{\lambda, \mu}_i = \frac{\lambda}{\lambda+A} X_i + \frac{A}{\lambda+A} \mu .
\end{equation}
The parametric empirical Bayes method for estimating the $\theta_i$'s consists
of first estimating the prior parameters $\lambda$ and $\mu$ and then
substituting them into~\eqref{eqn:map}. The prior parameters $\lambda$ and $\mu$
can be estimated by the MLE. For the special case of $\mu = 0$, this method
produces an estimator similar to the James-Stein estimator~\eqref{eqn:JS}.
Although this estimator is derived in an (empirical) Bayesian framework, it is of
interest to determine whether it has good frequentist properties. For example,
if we specify a loss function $L$ and consider the induced risk function $R$,
one would like to determine whether the estimator has uniformly smallest risk
within a reasonable class of estimators. In~\eqref{eqn:map}, the optimal choice
of $\lambda$ and $\mu$ is
\begin{equation*}
  \big(\hat{\lambda}^{\text{opt}}, \hat{\mu}^{\text{opt}}\big) =
  \argmin_{\lambda, \mu} R\big(\hat{\boldsymbol{\theta}}^{\lambda,
  \mu}, \boldsymbol{\theta}\big),
\end{equation*}
where $\boldsymbol{\theta} = [\theta_1,\ldots,\theta_p]^T$,
$\hat{\boldsymbol{\theta}}^{\lambda, \mu} = [\hat{\theta}^{\lambda,
\mu}_1,\ldots,\hat{\theta}^{\lambda, \mu}_p]^T$, and
$\hat{\lambda}^{\text{opt}}$ and $\hat{\mu}^{\text{opt}}$ depend on
$\boldsymbol{\theta}$, which is unknown. Instead of minimizing the risk function
directly, we will minimize Stein's unbiased risk estimate (SURE)
\citep{stein1981estimation}, denoted by $\text{SURE}(\lambda, \mu)$, which
satisfies $E_{\boldsymbol{\theta}}\left[ \text{SURE}(\lambda, \mu)\right] =
R(\hat{\boldsymbol{\theta}}^{\lambda, \mu}, \boldsymbol{\theta})$.
Thus, we will use
\[
\big(\hat{\lambda}^{\text{SURE}}, \hat{\mu}^{\text{SURE}}\big) =
\argmin_{\lambda, \mu} \text{SURE}(\lambda, \mu).
\]
The challenging part of this endeavor is to derive SURE, which depends heavily
on the risk function and the underlying distribution of the data. This approach
has been used to derive estimators for many models. For example,
\cite{xie2012sure} derived the (asymptotically) optimal shrinkage estimator for
a heteroscedastic hierarchical model, and their result is further generalized in
\cite{jing2016sure} and \cite{kong2017sure}.

\section{An Empirical Bayes Shrinkage Estimator for Log-Normal Distributions}\label{theory}

In this section, we consider the model 
\begin{equation*}
    X_{ij} \ind \text{LN}(M_i, \Sigma_i), \quad i=1,\ldots,p, \, j=1, \ldots, n,
\end{equation*}
and develop shrinkage estimators for the mean $\boldsymbol{M} = [M_1,\ldots,M_p]$ and the covariance matrix $\boldsymbol{\Sigma}=[\Sigma_1,\ldots,\Sigma_p]$. These $X_{ij}$'s are $P_N$-valued random matrices. For completeness, we first briefly review the shrinkage  estimator proposed by \citet{yang2019shrinkage} for $\boldsymbol{M}$ assuming $\Sigma_i=A_iI$, where the $A_i$'s are known positive numbers and $I$ is the identity matrix. The assumption on $\boldsymbol{\Sigma}$ is useful when $n$ is small since for small sample sizes the MLE for $\boldsymbol{\Sigma}$ is very unstable. Next, we present estimators for both $\boldsymbol{M}$ and $\boldsymbol{\Sigma}$. Besides presenting these estimators, we establish asymptotic optimality results for the proposed estimators. To be more precise, we show that the proposed estimators are asymptotically optimal within a large class of estimators containing the MLE. 

Another related interesting problem often encountered in practice involves group testing and estimating the ``difference'' between the two given groups. Consider the model
\begin{align*}
    X_{ij} & \ind \text{LN}(M^{(1)}_i, \Sigma^{(1)}_i), \quad i=1,\ldots,p,\,
    j=1,\ldots, n_x,\\
    Y_{ij} & \ind \text{LN}(M^{(2)}_i, \Sigma^{(2)}_i), \quad i=1,\ldots,p,\,
    j=1,\ldots, n_y,
\end{align*}
where the $X_{ij}$'s and $Y_{ij}$'s are independent. We want to estimate the differences between $M^{(1)}_i$ and $M^{(2)}_i$ for $i=1,\ldots,p$ and select the $i$'s for which the differences are large. However, the selected estimates tend to overestimate the corresponding true differences. The bias introduced by the selection process is termed by \emph{selection bias} \citep{dawid1994selection}. The selection bias originates from the fact that there are two possible reasons for the selected differences to be large: (i) the true differences are large and (ii) the random errors contained in the estimates are large. Tweedie's formula \citep{efron2011tweedie}, which we discuss and briefly review in section~\ref{sec:tweedie}, deals with precisely this selection bias, in the context of the normal means problem. In this work, we apply an analogue of Tweedie's formula designed for the context of SPD matrices.

\subsection{An Estimator of \texorpdfstring{$\boldsymbol{M}$}{Lg} When
\texorpdfstring{$\boldsymbol{\Sigma}$}{Lg} is Known}

For completeness, we briefly review the work of \citet{yang2019shrinkage} where the authors presented the estimator for $\boldsymbol{M}$ assuming that $\Sigma_i=A_iI$ where the $A_i$'s are known positive numbers. Under this assumption, they considered the class of estimators given by
\begin{equation} \label{eqn:postFM}
\widehat{M}_i^{\lambda,\mu} = \exp \left( \frac{n\lambda}{n\lambda+A_i} \log \bar{X}_i + \frac{A_i}{n\lambda + A_i} \log \mu \right),
\end{equation}
where $\mu \in P_N$, $\lambda > 0$, and $\bar{X}_i$ is the sample FM of $X_{i1},\ldots,X_{in}$. Using the Log-Euclidean distance as the loss function $L(\widehat{\boldsymbol{M}}, \boldsymbol{M}) = p^{-1}\sum_{i=1}^pd_{\text{LE}}^2(\widehat{M}_i,M_i)$, they showed that the SURE for the corresponding risk function $R(\widehat{\boldsymbol{M}}, \boldsymbol{M}) = EL(\widehat{\boldsymbol{M}}, \boldsymbol{M})$ is given by
\begin{align*}
\text{SURE}(\lambda, \mu) & = \frac{1}{p}\sum_{i=1}^p \frac{A_i}{(n\lambda +
A_i)^2} \left( A_i \Vert \! \log \bar{X}_i - \log \mu \Vert^2 +
\frac{q(n^2\lambda^2-A_i^2)}{n}\right).
\end{align*}
Hence, $\lambda$ and $\mu$ can be estimated by 
\begin{equation} \label{eqn:sure_opt}
    \big(\hat{\lambda}^{\text{SURE}}, \hat{\mu}^{\text{SURE}}\big) = \argmin_{\lambda, \mu} \text{SURE}(\lambda, \mu).
\end{equation} 
Their shrinkage estimator for $M_i$ is given by
\begin{equation} \label{eqn:SURE_est}
\widehat{M}_i^{\text{SURE}} = \exp \left(
\frac{n\hat{\lambda}^{\text{SURE}}}{n\hat{\lambda}^{\text{SURE}}+A_i} \log
\bar{X}_i + \frac{A_i}{n\hat{\lambda}^{\text{SURE}} + A_i} \log
\hat{\mu}^{\text{SURE}} \right).
\end{equation}
They also presented the following two theorems showing the asymptotic optimality of the shrinkage estimator.
\begin{theorem} \label{thm1}
Assume the following conditions:
\begin{enumerate}
    \item[(i)] $\limsup_{p\to\infty} p^{-1}\sum_{i=1}^p A_i^2 < \infty$,
    \item[(ii)] $\limsup_{p\to\infty} p^{-1}\sum_{i=1}^p A_i\Vert \! \log M_i \Vert^2 < \infty$,     \item[(iii)] $\limsup_{p\to\infty} p^{-1}\sum_{i=1}^p \Vert \! \log M_i \Vert^{2+\delta} < \infty$ for some $\delta > 0$.
\end{enumerate}
Then,
\[
\sup_{\lambda > 0, \Vert \! \log \mu \Vert < \max_i \Vert \! \log
\bar{X}_i \Vert} \vert \text{SURE}(\lambda, \mu) -
L(\widehat{\boldsymbol{M}}^{\lambda, \mu}, \boldsymbol{M})\vert
\stackrel{\text{prob}}{\longrightarrow} 0 \quad \text{as $p \to \infty$}.
\]
\end{theorem}
\begin{theorem} \label{thm2} 
If assumptions (i), (ii) and (iii) in Theorem \ref{thm1} hold, then
\[
\lim_{p \to \infty} [R(\widehat{\boldsymbol{M}}^{\text{SURE}}, \boldsymbol{M}) - R(\widehat{\boldsymbol{M}}^{\lambda, \mu}, \boldsymbol{M})] \leq 0 .
\]
\end{theorem}

\subsection{Estimators for \texorpdfstring{$\boldsymbol{M}$}{Lg} and \texorpdfstring{$\boldsymbol{\Sigma}$}{Lg}}
In \citet{yang2019shrinkage}, the covariance matrices of the underlying distributions were assumed to be known to simplify the derivation. In real applications however, the covariance matrices are rarely known, and in practice they must be estimated. In this paper, we consider the general case of unknown covariance matrices which is more challenging and pertinent in real applications. Let
\begin{align}
    \label{eqn:model}
    \begin{split}
    X_{ij} |(M_i, \Sigma_i) & \ind \text{LN}(M_i, \Sigma_i)\\
   M_i | \Sigma_i & \ind \text{LN}(\mu, \lambda^{-1}\Sigma_i)\\
   \Sigma_i & \iid \text{Inv-Wishart}(\Psi, \nu),
    \end{split}
\end{align}
for $i = 1,\ldots,p$ and $j = 1,\ldots,n$. The prior for $(M_i,\Sigma_i)$ is called the Log-Normal-Inverse-Wishart (LNIW) prior, and it is motivated by the normal-inverse-Wishart prior in the Euclidean space setting. We would like to emphasize that the main reason for choosing the LNIW prior over others is the property of conjugacy which leads to a closed-form expression for our estimators. Let 
\begin{align}
    \label{eqn:FM_and_cov}
    \bar{X}_i = \exp\bigg(n^{-1}\sum_{j=1}^n\log X_{ij}\bigg) \quad \text{and} \quad
    S_i = \sum_{j=1}^n \big( \widetilde{X}_{ij} - \widetilde{ \bar{X}_i}\big)\big( \widetilde{X}_{ij} - \widetilde{ \bar{X}_i}\big)^T.
\end{align}
Then the posterior distributions of $M_i$ and $\Sigma_i$ are given by
\begin{align*}
    M_i|\big(\{X_{ij}\}_{i,j},\{\Sigma_i\}_{i=1}^p\big) & \sim \text{LN}\Big(\exp\Big(\frac{n\log \bar{X}_i + \lambda \log \mu}{\lambda+n}\Big),(\lambda + n)^{-1}\Sigma_i\Big)\\
    \Sigma_i|S_i & \sim \text{Inv-Wishart}(\Psi+S_i, \nu + n - 1),
\end{align*}
and the posterior means for $M_i$ and $\Sigma_i$ are given by
\begin{align}
    \label{eqn:est_M_sigma}
    \widehat{M}_i = \exp\Big(\frac{n\log \bar{X}_i + \lambda \log \mu}{\lambda+n}\Big) \quad \text{and} \quad
    \widehat{\Sigma}_i = \frac{\Psi+S_{i}}{\nu+n-q-2}.
\end{align}

Consider the loss function 
\begin{align*}
L\big((\widehat{\bd{M}}, \widehat{\bd{\Sigma}}), (\bd{M},\bd{\Sigma})\big) & = p^{-1}\sum_{i=1}^pd_{\text{LE}}^2(\widehat{M}_i,M_i) + p^{-1}\sum_{i=1}^p\|\widehat{\Sigma}_i-\Sigma_i\|^2\\
& = L_1(\widehat{\bd{M}}, \bd{M}) + L_2(\widehat{\bd{\Sigma}}, \bd{\Sigma}).
\end{align*}
Its induced risk function is given by
\begin{align*}
    R\big((\widehat{\boldsymbol{M}},\widehat{\boldsymbol{\Sigma}}),(\boldsymbol{M},\boldsymbol{\Sigma})\big) & = p^{-1}\sum_{i=1}^p Ed^2_{\text{LE}}(\widehat{M}_i,M_i) + E\Vert \widehat{\Sigma}_i - \Sigma_i\Vert^2 \\
    & = p^{-1}(\lambda+n)^{-2}\sum_{i=1}^p\big[n\text{tr}\Sigma_{i}+\lambda^{2}d^2_{\text{LE}}(\mu, M_{i})\big]\\
    & \qquad + p^{-1}\sum_{i=1}^p (\nu+n-q-2)^{-2}\Big[\big(n-1+(\nu-q-1)^{2}\big)\text{tr}(\Sigma_{i}^{2})\\
    & \qquad - 2(\nu-q-1)\text{tr}(\Psi\Sigma_{i})+(n-1)(\text{tr}\Sigma_{i})^{2}+\text{tr}(\Psi^{2})\Big]
\end{align*}
with the detailed derivation given in the supplemental material. The SURE for this risk function is 
\begin{align*}
\text{SURE}(\lambda,\Psi,\nu, \mu) & = p^{-1}\Bigg\{\sum_{i=1}^{p}(\lambda+n)^{-2}\Big[\frac{n-\lambda^2/n}{n-1}\tr S_{i}+ \lambda^2d^2_{\text{LE}}(\bar{X}_i,\mu)\Big]\\
 & \qquad + (\nu+n-q-2)^{-2}\Big[\frac{n-3 + (\nu-q-1)^2}{(n+1)(n-2)}\tr(S^2_i)\\
 & \qquad + \frac{(n-1)^2-(\nu-q-1)^2}{(n-1)(n+1)(n-2)}\big(\tr S_i\big)^2 - 2\frac{\nu-q-1}{n-1}\tr(\Psi S_i) + \tr (\Psi^2)\Big]\Bigg\}
\end{align*}
with the detailed derivation given in the supplemental material.

The hyperparameter vector $(\lambda,\Psi,\nu,\mu)$ is estimated by minimizing $\text{SURE}(\lambda,\Psi,\nu,\mu)$, and the resulting shrinkage estimators of $M_i$ and $\Sigma_i$ are obtained by plugging in the minimizing vector into~\eqref{eqn:est_M_sigma}. Note that this is a non-convex optimization problem and for such problems, the convergence relies heavily on the choice of the initialization. We suggest the following initialization, which is discussed in the supplemental material:
\begin{align*}
    \mu_0 & = \exp\Bigg(p^{-1}\sum_{i=1}^p\log\bar{X}_i\Bigg)\\
    \lambda_0 & = \frac{np^{-1}\sum_{i=1}^pd^2_{\text{LE}}(\bar{X}_i, \mu_0)}{\frac{n}{p(n-1)}\sum_{i=1}^p\tr S_i - p^{-1}\sum_{i=1}^pd^2_{\text{LE}}(\bar{X}_i, \mu_0)}\\
    \nu_0 & = \frac{q+1}{\frac{n-q-2}{p^2q(n-1)}\tr\big[\big(\sum_{i=1}^p S_i\big)\big(\sum_{i=1}^pS_i^{-1}\big)\big] - 1} + q + 1\\
    \Psi_0 & = \frac{\nu_0-q-1}{p(n-1)}\sum_{i=1}^pS_i.
\end{align*}
In all our experiments, the algorithm converged in less than 20 iterations with the suggested initialization. This concludes the description of our estimators of the unknown means and covariance matrices. Theorem~\ref{thm3} below states that $\text{SURE}(\lambda, \Psi, \nu,\mu)$ approximates the true loss $L\big(\big(\widehat{\bd{M}}^{\lambda,\mu}, \widehat{\bd{\Sigma}}^{\Psi, \nu}\big), \big(\bd{M}, \bd{\Sigma}\big)\big)$ well in the sense that the difference between the two random variables converges to 0 in probability as $p \to \infty$. Additionally, Theorem~\ref{thm4} below shows that the estimators of $\bd{M}$ and $\bd{\Sigma}$ obtained by minimizing $\text{SURE}(\lambda,\Psi,\nu,\mu)$ are asymptotically optimal in the class of estimators of the form~\eqref{eqn:est_M_sigma}.

\begin{theorem}\label{thm3}
    Assume the following conditions:
\begin{enumerate}
    \item[(i)] $\limsup_{p\to\infty} p^{-1}\sum_{i=1}^p \big(\tr \Sigma_i\big)^4 < \infty$,
    \item[(ii)] $\limsup_{p\to\infty} p^{-1}\sum_{i=1}^p
        \widetilde{M}^T_i\Sigma_i\widetilde{M}_i < \infty$,
    \item[(iii)] $\limsup_{p\to\infty} p^{-1}\sum_{i=1}^p \Vert \! \log M_i \Vert^{2+\delta} < \infty$ for some $\delta > 0$.
\end{enumerate}
Then
\[
    \sup_{\substack{\lambda > 0, \nu > q+1, \|\Psi\|\leq \max_{1 \leq i \leq
                p}\|S_i\|,\\ \|\!\log\mu\|\leq \max_{1\leq i\leq
    p}\|\!\log\bar{X}_i\|}} \Big|\text{SURE}(\lambda, \Psi, \nu, \mu)-
    L\Big(\big(\widehat{\bd{M}}^{\lambda,\mu},
            \widehat{\bd{\Sigma}}^{\Psi, \nu}\big), (\bd{M},
    \bd{\Sigma})\Big)\Big| \stackrel{\text{prob}}{\longrightarrow} 0\quad
    \text{as }p \to \infty.
\]
\end{theorem}

Note that the optimization has some constraints. However, in practice, with proper initialization as suggested earlier, the constraints on $\Psi$ and $\mu$ can be safely ignored. The reason is that, for $\Psi$ and $\mu$ far from $S_i$'s and $\bar{X}_i$ respectively, the value of SURE will be large. The constraints on $\lambda$ and $\nu$ can easily be handled by standard constrained optimization algorithms, e.g.\ L-BFGS-B \citep{byrd1995limited}.

\begin{theorem}\label{thm4}
    If assumptions (i), (ii), and (iii) in Theorem~\ref{thm3} hold, then
    \begin{align*}
        \lim_{p\to\infty} \Bigg[ R\Big(\big( \widehat{\bd{M}}^{\text{SURE}},
        \widehat{\bd{\Sigma}}^{\text{SURE}} \big), (\bd{M},
        \bd{\Sigma})\Big) - R\Big(\big(
        \widehat{\bd{M}}^{\lambda,\mu}, \widehat{\bd{\Sigma}}^{\Psi,
        \nu} \big), (\bd{M}, \bd{\Sigma})\Big)\Bigg] \leq 0.       
    \end{align*}
\end{theorem}

Note that in all the above theorems, we consider the asymptotic regime $p \to \infty$ while the size of the SPD matrix $N$ is held fixed. The main reason for fixing the size of the SPD matrix is that in our application, namely the DTI analysis, the size of the diffusion tensors is always $3 \times 3$, because the diffusion magnetic resonance images are 3-dimensional images. However, the number of voxels $p$ can increase as they are determined by the resolution of the acquired image which can increase due to advances in medical imaging technology. This is different from the usual high dimensional covariance matrix estimation problem in which the size of the covariance matrix is allowed to grow.

{\bf Remark:} The proofs of Theorems~\ref{thm1} and \ref{thm2} in \cite{yang2019shrinkage} use arguments similar to those that already exist in the literature, and in that sense they are not very difficult. In contrast, the proofs of our Theorems~\ref{thm3} and \ref{thm4} above do not proceed along familiar lines. Indeed, they are rather complicated, the difficulty being that bounding the moments of Wishart matrices or the moments of trace of Wishart matrices is nontrivial when the orders of the required moments are higher than two.  We present these proofs in the supplementary document.

\subsection{Tweedie's Formula for \texorpdfstring{$F$}{Lg} Statistics}\label{sec:tweedie}

One of the motivations for the development of our approach for estimating $\boldsymbol{M}$ and $\boldsymbol{\Sigma}$ is a problem in neuroimaging involving detection of differences between a patient group and a control group. The problem can be stated as follows. There are $n_x$ patients in a disease group and $n_y$ normal subjects in a control group. We consider a region of the brain image consisting of $p$ voxels.  As explained in section~\ref{sec:real_data}, the local diffusional property of water molecules in the human brain is of clinical importance and it is common to capture this diffusional property at each voxel in the diffusion magnetic resonance image (dMRI) via a zero-mean Gaussian with a $3 \times 3$ covariance matrix. Using any of the existing state-of-the-art dMRI analysis techniques, it is possible to estimate, from each patient image, the diffusion tensor $M_i$ corresponding to voxel $i$, for $i = 1, \ldots, p$.  Let $M^{(1)}_i$ and $M^{(2)}_i$ denote the diffusion tensors corresponding to voxel $i$ for the disease and control groups respectively.  The goal is to identify the indices $i$ for which the difference between $M^{(1)}_i$ and $M^{(2)}_i$ is large.  The model we consider is
\begin{align*}
X_{ij} & \ind \text{LN}(M^{(1)}_i, \Sigma_i), \quad i = 1,\ldots,p, \; j =
1,\ldots, n_x,\\ 
Y_{ij} & \ind \text{LN}(M^{(2)}_i, \Sigma_i), \quad i = 1,\ldots,p, \; j =
1,\ldots, n_y.
\end{align*}

In this work, we compute the Hotelling $T^2$ statistic for each $i = 1,\ldots,p$ as a measure of the difference between $M^{(1)}_i$ and $M^{(2)}_i$. The Hotelling $T^2$ statistic for SPD matrices has been proposed by \citet{schwartzman2010group} and is given by
\begin{equation} \label{eqn:T2_stat}
t_i^{2}=\Big(\widetilde{\bar{X}_{i}}-\widetilde{\bar{Y}_{i}}\Big)^{T}\Big[\Big(\frac{1}{n_x}+\frac{1}{n_y}\Big)S_i\Big]^{-1}\Big(\widetilde{\bar{X}_{i}}-\widetilde{\bar{Y}_{i}}\Big)
\end{equation}
where $\bar{X}_i$ and $\bar{Y}_i$ are the FMs of $\{X_{ij}\}_j$ and $\{Y_{ij}\}_j$ and $S_i = (n_x+n_y-2)^{-1}\big(S^{(1)}_i + S^{(2)}_i\big)$ is the pooled estimate of $\Sigma_i$ where $S^{(1)}_i$ and $S^{(2)}_i$ are computed using~\eqref{eqn:FM_and_cov}. Since the $X_{ij}$'s and $Y_{ij}$'s are Log-normally distributed, one can easily verify that the sampling distribution of $t^2_i$ is given by
\[
\frac{\nu - q - 1}{\nu q}t_i^{2} \ind F_{q,\nu-q-1, \lambda_i}
\]
where $\nu = n_x + n_y - 2$ is a degrees of freedom parameter (recall that $q=N(N+1)/2$). Note that we make the
assumption that the covariance matrices for the two groups are the same, i.e.\
$\Sigma_i^{(1)}=\Sigma_i^{(2)}=\Sigma_i$. Similar results can be obtained for
the unequal covariance case, but with more complicated expressions for the $T^2$ statistics and the degrees of freedom parameters. The $\lambda_i$'s are the non-centrality parameters for the non-central $F$ distribution and are given by 
\[
\lambda_i = \Big(\frac{1}{n_x} + \frac{1}{n_y}\Big)\big(\widetilde{M}^{(1)}_i - \widetilde{M}^{(2)}_i\big)^T\Sigma_i^{-1}\big(\widetilde{M}^{(1)}_i - \widetilde{M}^{(2)}_i\big).
\]
These non-centrality parameters can be interpreted as the (squared) differences between $M^{(1)}_i$ and $M^{(2)}_i$, and they are the parameters we would like to estimate using the statistics~\eqref{eqn:T2_stat} computed from the data. Then, based on the estimates $\hat{\lambda}_i$, we select those $i$'s with large estimates, say the largest $1\%$ of all $\hat{\lambda}_i$. However, the process of selection from the computed estimates introduces a \emph{selection bias} \citep{dawid1994selection}. The selection bias comes from the fact that it is possible to select some indices $i$'s for which the actual $\lambda_i$'s are not large but the random errors are large, so that the estimates $\hat{\lambda}_i$ are pushed away from the true parameters $\lambda_i$. There are several ways to correct this selection bias, and \cite{efron2011tweedie} proposed to use \emph{Tweedie's formula} for such a purpose. 

Tweedie's formula was first proposed by \cite{robbins1956}, and we review this formula here in the context of the classical normal means problem, which is stated as follows. We observe $Z_i \ind N(\mu_i, \sigma^2)$, $i=1,\ldots,p$, where the $\mu_i$'s are unknown and $\sigma^2$ is known, and the goal is to estimate the $\mu_i$'s. In the empirical Bayes approach to this problem we assume that $\mu_i$'s are iid according to some distribution $G$. The marginal density of the $Z_i$'s is then $f(z)=\int\phi_{\sigma}(z-\mu)\, dG(\mu)$, where $\phi_{\sigma}$ is the density of the $N(0, \sigma^2)$ distribution. With this notation, if $G$ is known (so that $f$ is known), the best estimator of $\mu_i$ (under squared error loss) is the so-called Tweedie estimator given by
\[
\hat{\mu}_i = Z_i + \sigma^2\frac{f^{\prime}(Z_i)}{f(Z_i)}.
\]
A feature of this estimator is that it depends on $G$ only through $f$, and this is desirable because it is fairly easy to estimate $f$ from the $Z_i$'s (so we don't need to specify $G$). Another interesting observation about this estimator is that $\hat{\mu}_i$ is shrinking the MLE $\hat{\mu}_i^{\text{MLE}} = Z_i$ and can be viewed as a generalization of the James-Stein estimator, which assumes $\mu_i \iid N(0, \lambda)$ with unknown $\lambda$. The Tweedie estimator can be generalized to exponential families. Suppose that $Z_i|\eta_i \ind f_{\eta_i}(z) = \exp(\eta_iz-\phi(\eta_i))f_0(z)$, and the prior for $\phi$ is $G$. Then the Tweedie estimator for $\eta_i$ is 
\[
\hat{\eta}_i = l^{\prime}(Z_i)-l^{\prime}_0(Z_i),
\]
where $l(z) = \log\int f_{\eta}(z)\, dG(\eta)$ is the log of the marginal likelihood of the $Z_i$'s and $l_0(z)=\log f_0(z)$.

Although this formula is elegant and useful, it applies only to exponential families. Recently, \cite{du2020empirical} derived a Tweedie-type formula for non-central $\chi^{2}$ statistics, for situations where one is interested in estimating the non-centrality parameters. Suppose $Z_i|\lambda_i \ind \chi_{\nu,\lambda_i}^{2}$ and $\lambda_i\iid G$. Then, 
\begin{equation} \label{eqn:tweedie_ncChi}
E(\lambda_i|Z_i)=\Big[(Z_i-\nu+4)+2Z_i\Big(\frac{2l_{\nu}^{\prime\prime}(Z_i)}{1+2l_{\nu}^{\prime}(Z_i)}+l_{\nu}^{\prime}(Z_i)\Big)\Big]\big(1+2l_{\nu}^{\prime}(Z_i)\big),
\end{equation}
where $l_{\nu}(\cdot)$ is the marginal log-likelihood of the $Z_i$'s (see Theorem 1 in \cite{du2020empirical}).

For our situation, if we define $Z_i = [(\nu-q-1)/\nu q]t_i^2$, then 
\[
Z_i|\lambda_i \ind F_{q,\nu - q - 1, \lambda_i}
\] 
(recall that $\nu=n_x + n_y - 2$). Assume that the $\lambda_i$'s are iid according to some distribution $G$. We would now like to address the problem of how to obtain an empirical Bayes estimate of $\lambda_i$. Let $\Phi_{\nu_1, \nu_2, \lambda}$ be the cumulative distribution function (cdf) of the non-central $F$ distribution, $F_{\nu_1,\nu_2,\lambda}$, and let $\Tilde{\Phi}_{\nu, \lambda}$ be the cdf of the non-central $\chi^2$ distribution, $\chi^2_{\nu, \lambda}$. Then the transformed variable $Y_i = \Tilde{\Phi}^{-1}_{\nu_1,\lambda_i}(\Phi_{\nu_1,\nu_2,\lambda_i}(Z_i))$ follows a non-central $\chi^2$ distribution with degrees of freedom parameter $\nu_1$ and non-centrality parameter $\lambda_i$, and we note that when $\nu_2$ is large, $\Phi_{\nu_1,\nu_2,\lambda_i}$ and $\tilde{\Phi}_{\nu_1,\lambda_i}$ are nearly equal, so that this quantile transformation is nearly the identity. However, the transformation depends on $\lambda_i$, which is the parameter to be estimated, so we propose the following iterative algorithm for estimating $E(\lambda_i|Z_i)$. Let $\lambda_i^{(t)}$ be the estimate of $\lambda_i$ at the $t$-th iteration. Then, our iterative update of $\lambda_i$ is given by 
\begin{equation*}
\lambda_i^{(t+1)}=\Big[(Y_i^{(t)}-\nu_1+4)+2Y_i^{(t)}\Big(\frac{2l_{\nu_1}^{\prime\prime}(Y_i^{(t)})}{1+2l_{\nu_1}^{\prime}(Y_i^{(t)})}+l_{\nu_1}^{\prime}(Y_i^{(t)})\Big)\Big]\big(1+2l_{\nu_1}^{\prime}(Y_i^{(t)})\big),
\end{equation*}
where $Y_i^{(t)} = \tilde{\Phi}^{-1}_{\nu_1,\lambda_i^{(t)}} \big(\Phi_{\nu_1,\nu_2,\lambda_i^{(t)}}(Z_i)\big)$, $\nu_1=q$, and $\nu_2=\nu-q-1$. Now the marginal log-likelihood $l_{\nu_1}(y)$ is not available since the prior $G$ for $\lambda_i$ is unknown.  There are several ways to estimate the marginal density of the $Y_i^{(t)}$'s. One of these is through kernel density estimation.  However, the iterative formula involves the first and second derivatives of the marginal log-likelihood, and estimates of the derivatives of a density produced through kernel methods are notoriously unstable (see Chapter 3 of \cite{silverman1986density}).  There exist different approaches to deal with this problem (see \citet{sasaki2016direct} and \citet{shen2017posterior}).  Here we follow \citet{efron2011tweedie} and postulate that $l_{\nu_1}$ is well approximated by a polynomial of degree $K$, and write $l_{\nu_1}(y) = \sum_{k=0}^K \beta_k y^k$.  The coefficients $\beta_k$, $k=1,\ldots,K$, can be estimated via \emph{Lindsey's method} \citep{efron1996using}, which is a Poisson regression technique for (parametric) density estimation; the coefficient $\beta_0$ is determined by the requirement that $f_{\nu_1}(y) =\exp(l_{\nu_1}(y))$ integrates to $1$. The advantage of Lindsey's method over methods that use kernel density estimation is that it does not require us to estimate the derivatives separately, since $l^{\prime}_{\nu_1}(y)=\sum_{k=1}^Kk\beta_k y^{k-1}$ and $l^{\prime\prime}_{\nu_1}(y)=\sum_{k=2}^Kk(k-1)\beta_k y^{k-2}$. In our experience, with $l_{\nu_1}'$ and $l_{\nu_1}''$ estimated in this way, if we initialize the scheme by setting $\lambda_i^{(0)}$ to be the estimate of $\lambda_i$ given by \cite{du2020empirical} procedure, then the algorithm converges in less than 10 iterations.

\section{Experimental Results}\label{exp}

In this section, we describe the performance of our methods on two synthetic
data sets and on two sets of real data  acquired using diffusion magnetic
resonance brain scans of normal (control) subjects and patients with Parkinson's
disease. For the synthetic data experiments, we show that (i) the proposed
shrinkage estimator for the FM (SURE.Full-FM; \emph{with} simultaneous
estimation of the covariance matrices) outperforms the sample FM (FM.LE) and the
shrinkage estimator proposed by \cite{yang2019shrinkage} (SURE-FM; \emph{with
fixed} covariance matrices) (section \ref{sec:synth1}); (ii) the shrinkage
estimates of the differences capture the regions that are significantly
different between two groups of SPD matrix-valued images (section~\ref{sec:synth2}). For the real data experiments, we demonstrate that (iii) the
SURE.Full-FM provides improvement over the two competing estimators (FM.LE, and SURE-FM) for computing an atlas (templates) of diffusion tensor
images acquired from human brains (section~\ref{sec:real1}); (iv) the
proposed shrinkage estimator for detecting group differences is able to identify
the regions that are different between patients with Parkinson's disease and control subjects
(section~\ref{sec:real2}). Details of these experiments will be presented
subsequently.  

\subsection{Synthetic Data Experiments} \label{sec:synthetic}

We present two synthetic data experiments here to show that the proposed
shrinkage estimator, SURE.Full-FM, outperforms the sample FM and SURE-FM and
that the shrinkage estimates of the group differences can accurately localize
the regions that are significantly different between the two groups.

\subsubsection{Comparison Between SURE.Full-FM and Competing Estimators} \label{sec:synth1}

Using generated noisy SPD fields ($P_3$) as data, here we present performance
comparisons of four estimators of $\bd{M}$: (i) SURE.Full-FM, which is the
proposed shrinkage estimator, (ii) SURE-FM proposed by \cite{yang2019shrinkage}
which assumes known covariance matrices and (iii) the MLE, which is denoted by
FM.LE, since by proposition~\ref{prop:MLE} it is the FM based on the
Log-Euclidean metric. The synthetic data are generated according to \eqref{eqn:model}.
Specifically, we set $\mu = I_3$, $\Psi = I_6$, and $n=10$, and we vary the
variance $\lambda$ and the degree of freedom $\nu$ of the prior distribution as
follows: $\lambda=10, 50$, and $\nu=15, 30$. Figure~\ref{fig:risk} depicts the
relationship between the average loss (averaged over $m=1000$ replications) and
the dimension $p$ under varying conditions for the four estimators. Note that
since the covariance matrices $\Sigma_i$'s are unknown in our synthetic
experiment and $(n-1)^{-1}S_i$ is an unbiased estimate for $\Sigma_i$, the
$A_i$'s in~\eqref{eqn:SURE_est} can be unbiasedly estimated by $[(n-1)q]^{-1}\tr
S_i$.

\begin{figure}[ht!]
\centering
\includegraphics[scale=0.9]{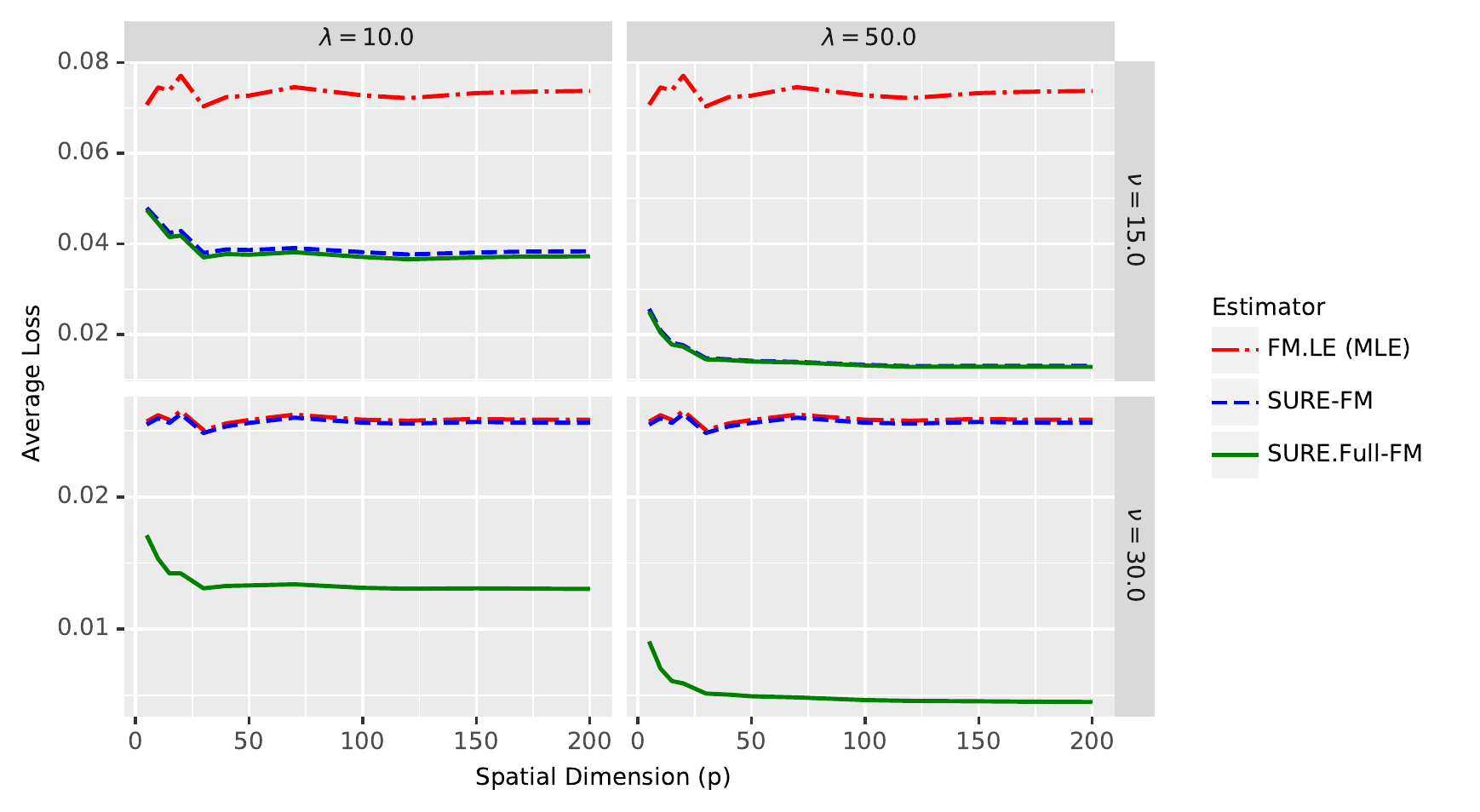}
\caption{Average loss for the three
estimators. Results for varying $\lambda$ and degree of freedom $\nu$ are shown
across the columns and rows, respectively. Note that in the bottom two panels,
the line corresponding to FM.LE is essentially the same as the SURE.FM, but is barely visible.}
\label{fig:risk}
\end{figure}

As is evident from Figure~\ref{fig:risk}, for large $\lambda$ the gains from
using SURE.Full-FM are greater. This observation is in accordance with our
intuition, which is that for large $\lambda$, the $M_i$'s are clustered, and it
is beneficial to shrink the MLEs of the $M_i$'s towards a common value. The main
difference between SURE-FM and SURE.Full-FM is that the former requires
knowledge of the $\Sigma_i$'s and in general such information is not available,
and estimates for the $\Sigma_i$'s are needed to compute the SURE-FM. Hence the
performance of SURE-FM depends heavily on how good the estimates for the
$\Sigma_i$'s are. In our synthetic data experiment,  we consider the unbiased
estimate $\hat{A}_i = [(n-1)q]^{-1}\tr S_i$ for SURE-FM. In this case, the prior
mean for $\Sigma_i$ is $E(\Sigma_i) = (\nu-q-1)^{-1}I_q$ for which the
assumption $\Sigma_i = A_iI$ seems reasonable. For large $\nu$, the
$\hat{A}_i$'s are closer to zero, which results in a smaller shrinkage effect
(this can be observed in Figure~\ref{fig:risk}, where we see that SURE-FM is
almost identical to FM.LE for $\nu=30$). \emph{Note that even if the assumption
$\Sigma_i=A_iI$ is not severely violated, our estimator still outperforms
SURE-FM by a large margin.}

On the other hand, we can fix $\lambda$ and $\nu$ to see how different choices
of $\mu$ and $\Psi$ affect the performance of our shrinkage estimator
SURE.Full-FM. To do this, we fix $n=10$, $\lambda=10$, and $\nu=15$ (so that we
can compare with the top-left panel of Figure~\ref{fig:risk}). We consider $\mu
= \text{diag}(2,0.5,0.5)$ and $\Psi_{ij} = 0.5^{|i-j|}$. The result is shown in
Figure~\ref{fig:risk2}. The top-left panel of Figure~\ref{fig:risk} shows that
when $\mu=I$ and $\Psi=I$, there is no difference between SURE-FM and
SURE.Full-FM, but Figure~\ref{fig:risk2} shows that when one of $\mu$ and $\Psi$
is not identity, our shrinkage estimator outperforms SURE-FM. For different
choices of $\lambda$ and $\nu$, the improvement will be more significant,
following the trend we observed in Figure~\ref{fig:risk}.

\begin{figure}[ht!]
\centering
\includegraphics[scale=0.9]{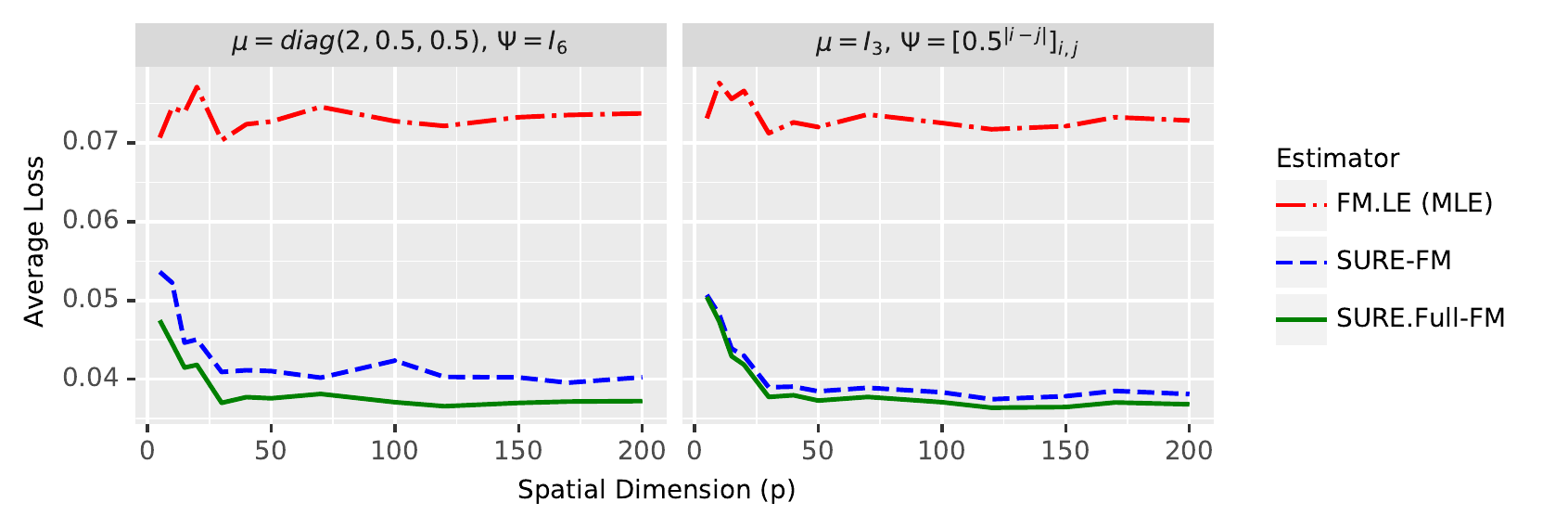}
\caption{Average losses for the four estimators. The left panel assumes $\mu =
    \text{diag}(2,0.5,0.5)$ and $\Psi = I$ and the right panel assumes $\mu = I$ and
    $\Psi_{ij} = 0.5^{|i-j|}$.}
\label{fig:risk2}
\end{figure}
 
\subsubsection{Differences Between Two Groups of SPD-Valued Images} \label{sec:synth2}

In this subsection, we demonstrate the method proposed in section
\ref{sec:tweedie} for evaluating the difference between two groups of SPD-valued
images. For this synthetic data experiment, we use $P_2$, the manifold of
$2\times 2$ SPD matrices, since it is easy to visualize these matrices. For the
visualization, we represent each $2 \times 2$ SPD matrix of the SPD-valued image
by an ellipse with the two eigenvectors as the axes of the ellipse  and the two
eigenvalues as the width and height along the corresponding axes respectively.
The data are generated as follows. Given $n_k$, $M^{(k)}_{i}$, $\sigma^2_i$,
$k=1,2$, $i=1,\ldots,p$, generate
\begin{align*}
    X_{ij} & \ind \text{LN}(M^{(1)}_{i}, \sigma^2_iI),\quad j=1,\ldots,n_1,\\
    Y_{ij} & \ind \text{LN}(M^{(2)}_{i}, \sigma^2_iI),\quad j=1,\ldots,n_2.
\end{align*}
We generate $n_1=n_2=30$ $P_2$-valued images for the two groups, and the size of
each $P_2$-valued image is $20 \times 20$, which gives $p = 20\times20=400$. For
the variances $\sigma_j$, we consider a low variance scenario $\sigma_i \iid
U(0.1,0.3)$ and a high variance scenario $\sigma_i \iid U(0.3,0.8)$. The means
$M^{(k)}_i$ are depicted visually in Figure~\ref{fig:group_mean} (in the form of
images with ellipses instead of gray values at each pixel), and the region in
which the means are different is the top-right corner, containing a quarter of
the pixels; this is the `ground truth' data. 

\begin{figure}[ht!]
\centering
\subfigure[Group 1]{
  \includegraphics[width=0.4\textwidth]{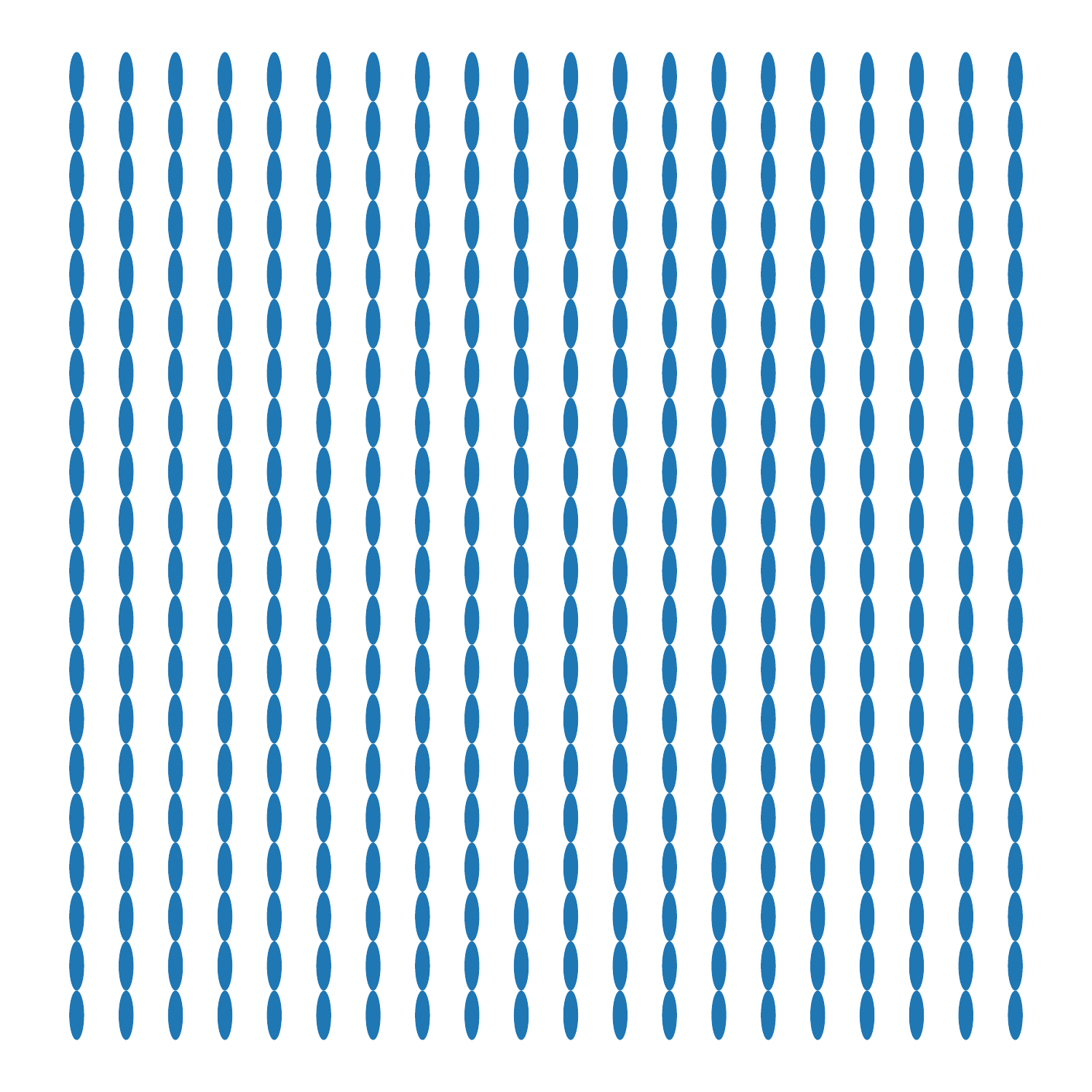}
}
\hfill
\subfigure[Group 2]
{\includegraphics[width=0.4\textwidth]{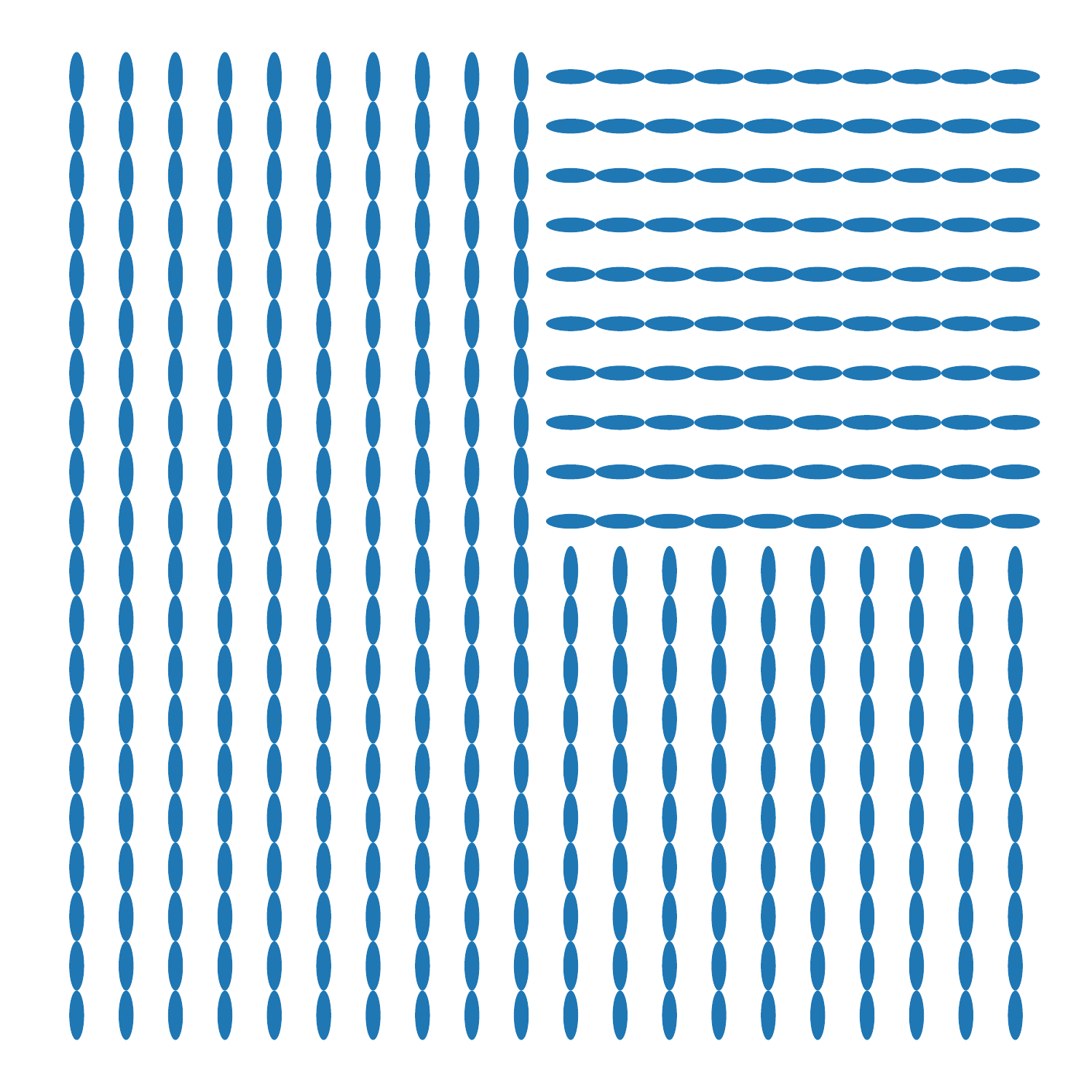}}
\caption{The mean $P_2$-valued images $M^{(k)}_i$, $k=1,2$, used to generate random  $P_2$-valued images for the two groups. The vertical ellipse represents the matrix $\diag(0.3, 1)$ and the horizontal ellipse represents the matrix $\diag(1,0.3)$.}
\label{fig:group_mean}
\end{figure}

As described in section \ref{sec:tweedie}, we first compute the Hotelling $T^2$
statistic from $\{X_{ij}\}_{j=1}^{n_1}$ and $\{Y_{ij}\}_{j=1}^{n_2}$ for each
$i$ and transform each of them to the $F$ statistic. Now we have $p$ non-central
$F$ statistics, $f_i \ind F_{\nu_1,\nu_2,\lambda_i}$, $i=1,\ldots,p$, where
$\nu_1=q=3$, and $\nu_2=n_1+n_2-2-q-1=n_1+n_2-6$. With the resulting $F$
statistics, we can apply the algorithm described in section~\ref{sec:tweedie} to
estimate the non-centrality parameters (at each location), and for the
estimation of the marginal log likelihood, we adopt Lindsey's method to fit a
polynomial of degree $K=5$ to the log-likelihood $l_{\nu_1}$. We have
experimented using different values of $K$, and we found that the results are
robust to changes in $K$, at least for relatively small $K$.  In our
experiments, we set $n_1=n_2=30$.  As we can see from
Figure~\ref{fig:group_mean}, we expect the method to yield large values on the
top-right corner of the image and small values for the rest of the matrix-valued
image (field). We compare the proposed estimator
$\hat{\lambda}_i^{\text{Tweedie}}$ to the estimator
$\hat{\lambda}_i^{\text{MOM}} = \max\big( \frac{\nu_1(\nu_2 - 2)}{\nu_2}f_i -
\nu_1, 0\big)$, which is obtained by the method of moments (MOM) and truncated
at 0, and also compare them for different $\sigma^2_i$'s. Note that we choose to
compare with the MOM estimator instead of the MLE for two reasons: (i) the MLE
for the non-centrality parameter of non-central $F$ distribution is expensive to
compute, and (ii) the MOM is commonly used as a standard for comparison, see for
example \cite{kubokawa1993estimation}. As we can see from the results, there are more black spots (which indicate no difference) in the top-right corner for the MOM estimates than the Tweedie-adjusted estimates. These show that the Tweedie-adjusted estimator allows us to capture the true region of difference
better than the MOM estimator does, especially for large
$\sigma^2_i$'s. This is due to the presence of the shrinkage effect in Tweedie's
formula. 

\begin{figure}[ht!]
\centering
\subfigure[$\hat{\lambda}^{\text{Tweedie}}$, $\sigma_i \iid U(0.1,0.3)$]
{\includegraphics[width=0.465\textwidth]{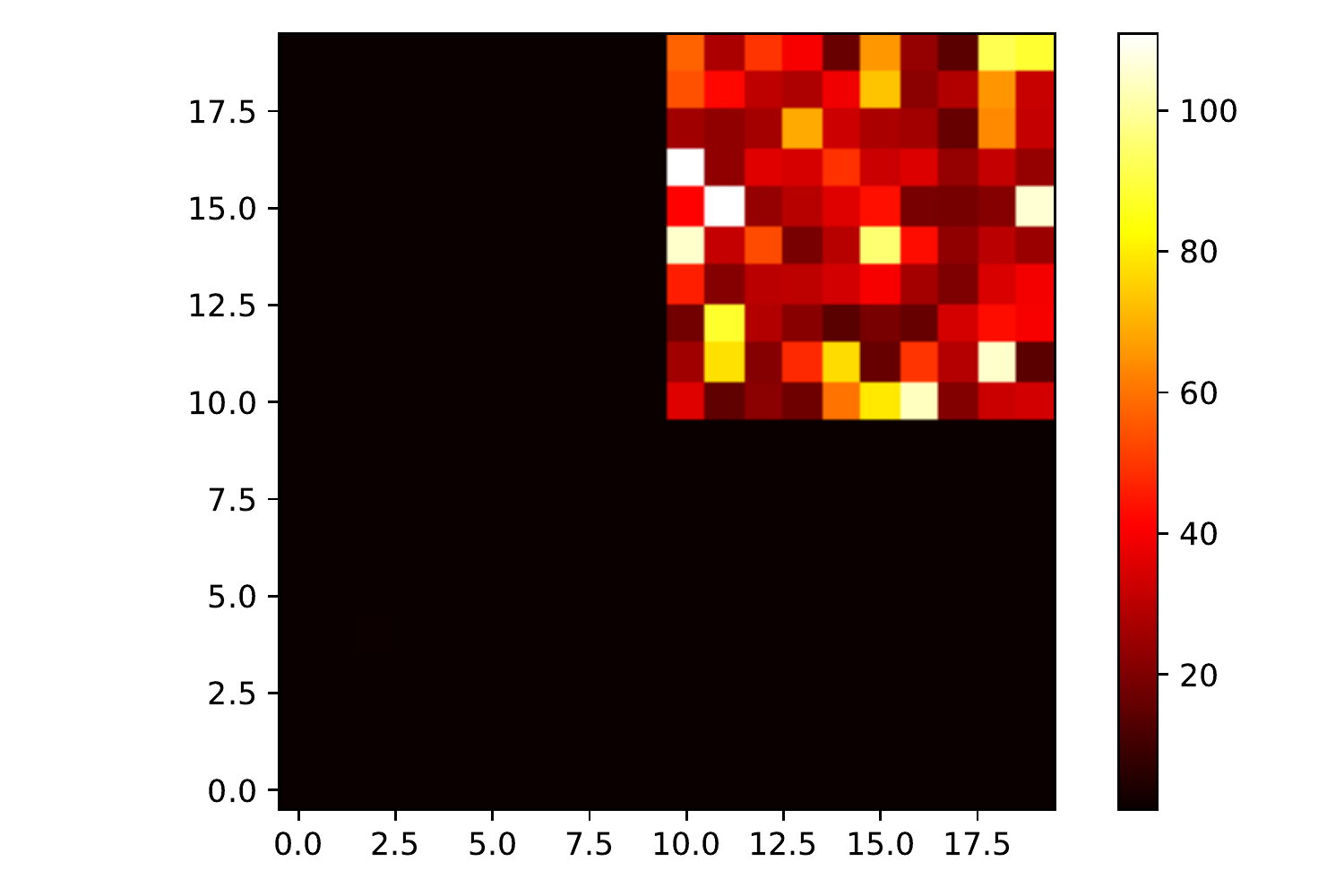}}
\hfill
\subfigure[$\hat{\lambda}^{\text{Tweedie}}$, $\sigma_i \iid U(0.3,0.8)$]
{\includegraphics[width=0.465\textwidth]{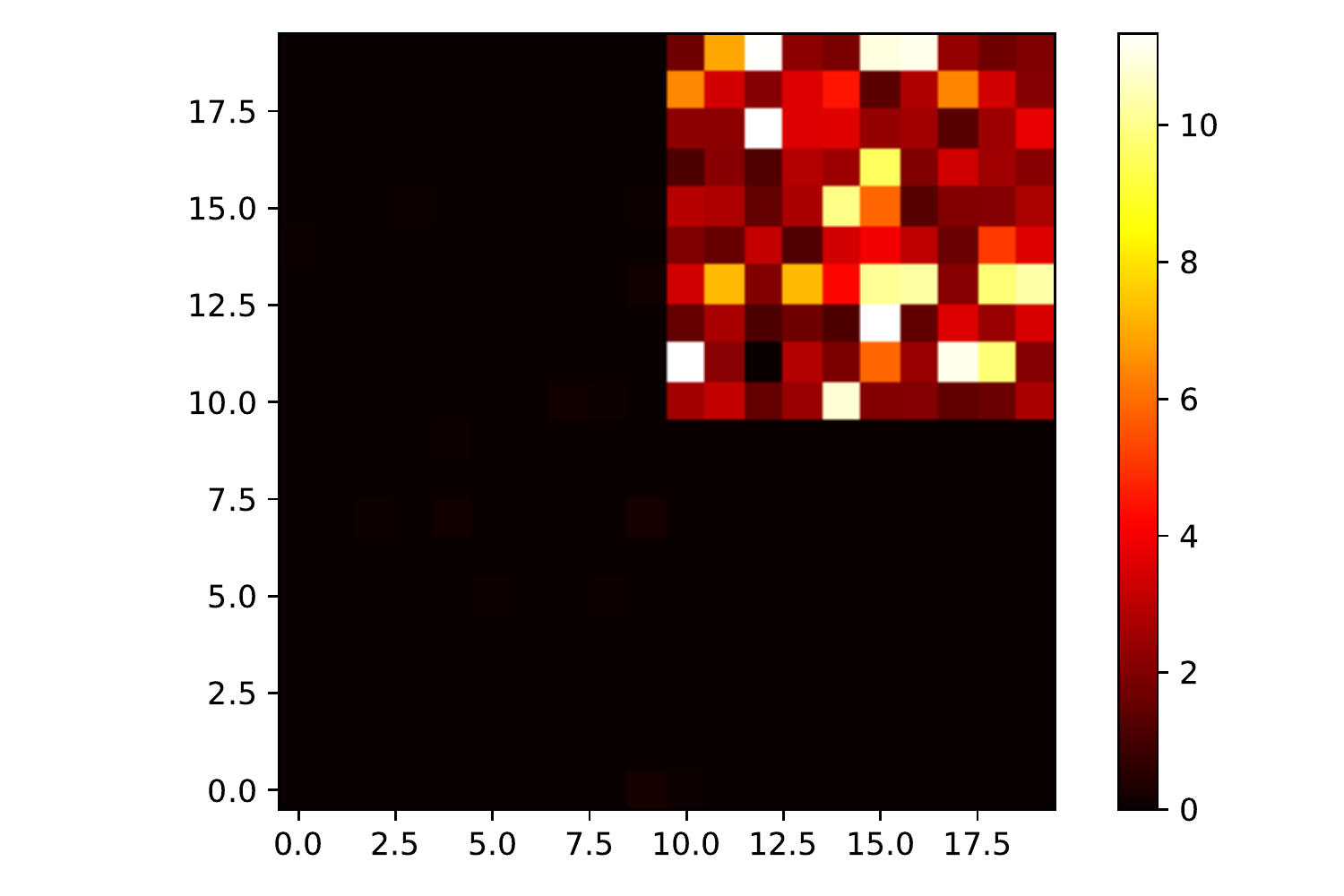}}
\\[7mm]
\subfigure[$\hat{\lambda}^{\text{MOM}}$, $\sigma_i \iid U(0.1,0.3)$]
{\includegraphics[width=0.465\textwidth]{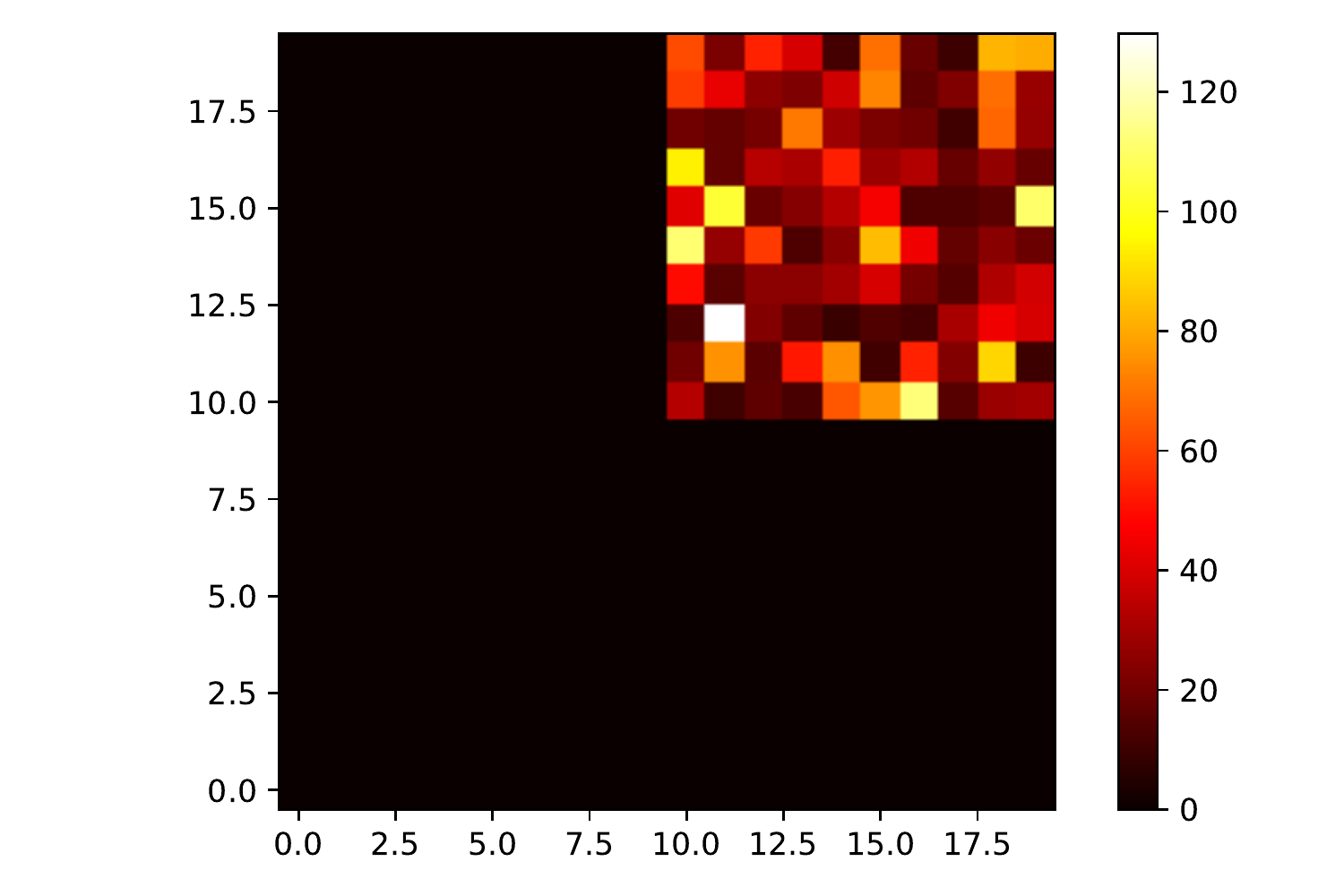}}
\hfill
\subfigure[$\hat{\lambda}^{\text{MOM}}$, $\sigma_i \iid U(0.3,0.8)$]
{\includegraphics[width=0.465\textwidth]{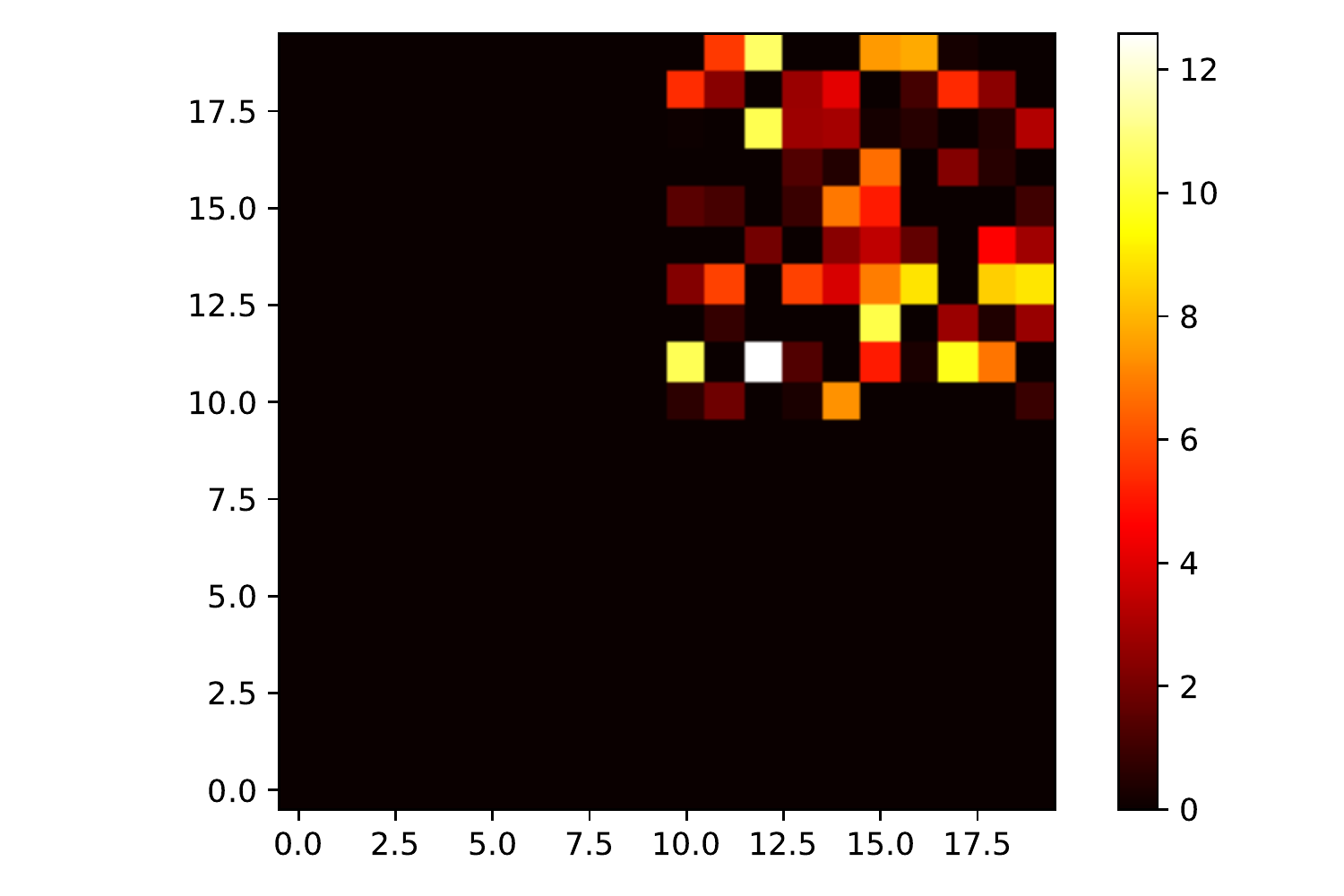}}
\hfill
\caption{Comparison between the proposed shrinkage estimates and MOM estimates of the non-centrality parameters.}\label{fig:tweedie_synthetic}
\end{figure}

\subsection{Real Data Experiments}\label{sec:real_data}

In this section, we present two real data experiments involving dMRI data sets. The diffusion MRI data used here is
available for public access via \url{https://pdbp.ninds.nih.gov/our-data}. dMRI
is a diagnostic imaging technique that allows one to non-invasively probe the
axonal fiber connectivity in the body by making the magnetic resonance signal sensitive to water diffusion through the tissue being imaged. In dMRI,
the water diffusion is fully characterized by the probability density function
(PDF) of the displacement of water molecules, called the ensemble average
propagator (EAP) \citep{callaghan1993principles}. A simple model that has been
widely used to describe the displacement of water molecules is a zero mean
Gaussian; its covariance matrix defines the diffusion tensor and characterizes
the diffusivity functional locally. The diffusion tensors are $3 \times 3$ SPD
matrices and hence have $6$ unique entries that need to be determined. Thus, the
diffusion imaging technique employed in this case involves the application of at
least $6$ diffusion sensitizing magnetic gradients for acquisition of full 3D MR
images \citep{basser1994mr}. This dMRI technique is called diffusion tensor
imaging (DTI). Some practical techniques for estimating the diffusion tensors
and population mean of diffusion tensors accurately can be found in
\cite{wang2004tensor}, \cite{chefd2004regularizing},
\cite{fletcher2004principal}, \cite{alexander2005multiple},
\cite{zhou2008bayesian}, \cite{lenglet2006dti}, and \cite{dryden2009non}. 

DTI has been the de facto non-invasive dMRI diagnostic imaging technique of
choice in the clinic for a variety of neurological ailments. After
fitting/estimating the diffusion tensors at each voxel, scalar-valued or
vector-valued measures are derived from the diffusion tensors for further
analysis. For instance, fractional anisotropy (FA) is a scalar-valued function
of the eigenvalues of the diffusion tensor and it was found that FA was reduced
in the neuro-anatomical structure called the Substantia Nigra in patients with
Parkinson's disease compared to control subjects \citep{vaillancourt2009high}. In
\cite{schwartzman2008false}, the authors proposed to use the principal
direction, which is the eigenvector corresponding to the largest eigenvalue of
the diffusion tensor, to represent the entire tensor; the principal direction
contains directional information that any scalar measures such as the FA does
not and hence, might be able to capture some subtle differences in the
anatomical structure of the brain. In this work, we use the full diffusion
tensor which captures both the eigen-values and eigen-vectors, in order to
assess the changes caused by pathologies to the underlying tissue
micro-architecture revealed via dMRI.

\subsubsection{Estimation of the Motor Sensory Tracts of Patients with Parkinson's Disease} \label{sec:real1}

In this section, we demonstrate the performance of SURE.Full-FM on the dMRI
scans of human brain data acquired from 50 patients with Parkinson's disease and 44 control
(normal) subjects. The diffusion MRI acquisition parameters were as
follows: repetition time = 7748ms, echo time = 86ms, flip angle = $90^{\circ}$,
number of diffusion gradients = 64, field of view = $224 \times 224$ mm,
in-plane resolution = 2 mm isotropic, slice-thickness = 2 mm, and SENSE factor =
2. All the dMRI data were pre-registered into a common coordinate frame prior to
any further data processing. 

The motor sensory area fiber tracts (M1 fiber tracts) are extracted from each
patient of the two groups using the FSL software
\citep{behrens2007probabilistic}. The size (length) of each tract is 33 voxels
for the left hemisphere tract and 34 voxels for the right hemisphere tract,
respectively. Diffusion tensors are then fitted to each of the voxels along each
of the tracts to obtain $p = 33$ ($p = 34$) $3 \times 3$ SPD matrices.  We then compute the
Log-Euclidean FM tract for each group. The FM tract here also has
33 (34) diffusion tensors along the tract. We will use these FMs computed from
the full population of each group as the `ground truth'; thus, the underlying
distribution in this experiment is the empirical distribution formed by the
observed data, i.e.\ the 33 (34) SPD matrices. Then, we randomly draw a
subsample of size $n = 10, 20, 50, 100$ (with replacement) respectively from
each group and compute the SURE.Full-FM (our proposed estimator)
and the three competing estimators (FM.LE and SURE-FM respectively) of
each group for each subsample size $n$. We compare the performance of the
different estimators by the Log-Euclidean distance between
the estimator and the `ground truth' FMs. The entire procedure is repeated for
$m=100$ random draws of subsamples and the average distances are reported in
Table~\ref{tab:dMRI}. Since our proposed shrinkage estimator jointly estimates
the FM and the covariance matrices, we also compare our covariance estimates,
denoted SURE.Full-Cov, with the MLE of the covariance matrices, i.e.,\ the
sample covariance matrices. The results are shown in Table~\ref{tab:dMRI_Sig}.

\begin{table}[ht!]
    \centering
    \setlength{\tabcolsep}{1mm}
\caption{Average loss for the three estimators in estimating the population FM for varying $n$ (with the standard errors in parentheses).}
\label{tab:dMRI} 
\begin{tabular}{lcccc}
    \toprule
    $n$     & 10                    & 20                   & 50 
            & 100\\\midrule
    FM.LE   & 0.774 (0.03)          & 0.405 (0.01)         & 0.159 (0.005)
            & 0.079 (0.002) \\
    SURE-FM & 0.772 (0.03)          & 0.404 (0.01)         & 0.160 (0.005) 
            & 0.079 (0.002)\\
    SURE.Full-FM & \textbf{0.199 (0.01)}   & \textbf{0.151 (0.003)}        & \textbf{0.094 (0.002)}
            & \textbf{0.057 (0.001)}\\
    \bottomrule
\end{tabular} 
\end{table} 

\begin{table}[ht!]
    \centering
    \setlength{\tabcolsep}{1mm}
\caption{Average loss for the two estimators, MLE and SURE.Full-Cov, in estimating the population covariance matrices for varying $n$ (with the standard errors in parentheses).}
\label{tab:dMRI_Sig} 
\begin{tabular}{lccccc}
    \toprule
    $n$           & 10                      & 20                   & 50             
                  & 100\\\midrule
    MLE           & 123.69 (5.71)           & 66.80 (2.69)         & 25.54 (0.91)   
                  & 12.91 (0.41) \\
    SURE.Full-Cov & \textbf{78.77 (3.21)}   &\textbf{52.02 (2.03)} & \textbf{22.81 (0.80)}   
                  & \textbf{12.15 (0.38)}\\
    \bottomrule
\end{tabular} 
\end{table} 

As is evident from Table~\ref{tab:dMRI}, the SURE.Full-FM outperforms the
competing estimators under varying size of subsamples. Also note that, as the
sample size increases, the improvement is less significant, which is consistent
with the observations on the synthetic data experiments in
section~\ref{sec:synthetic}. Recall that in section~\ref{sec:synth1}, the
SURE.FM and the SURE.Full-FM  perform equally well when the assumption $\Sigma_i
= A_iI$ is not violated severely. For real data, it is impossible to check this
assumption and it is unlikely to be true. Hence, in this real data experiment,
SURE.Full-FM outperforms SURE-FM by a large margin. The improvement of the
proposed shrinkage estimator for the covariance matrices over the MLEs is
evident from Table~\ref{tab:dMRI_Sig}.

\subsubsection{Tweedie-Adjusted Estimator as an Imaging Biomarker}\label{sec:real2}

Finally, we apply the shrinkage estimator proposed in section~\ref{sec:tweedie}
to identify the regions that are significantly distinct in diffusional
properties (as captured via diffusion tensors) between patients with Parkinson's disease and control
subjects. In this experiment, the dataset consists of DTI scans of 46 patients with Parkinson's disease and 24 control subjects. To identify the
differences between the two groups, we use the DTI of the
whole brain, which contains $p = 112 \times 112 \times 60$ voxels, without
pre-selecting any region of interest. The diffusion tensors are fitted at each
voxel across the whole brain volume. The goal of this experiment is to see if we
are able to automatically identify the regions capturing the large differences
between the Parkinson's disease group and control groups and qualitatively validate our findings
against what is expected by expert neurologists. In this context,
\cite{prodoehl2013diffusion} observed that the region most affected by
Parkinson's disease is the Substantia Nigra, which is contained in the Basal
Ganglia region of the human brain. 

After computing both the Tweedie-adjusted estimates and the MOM estimates of the
non-centrality parameters, we select voxels with the largest 1\% estimates of
the non-centrality parameters and mark those voxels in bright red. (Note that
there are other ways to determine the threshold for the selection, for example
by using the false discovery rate (FDR) in hypothesis testing problems. However,
this is beyond the scope of this paper and we refer the reader to
\cite{schwartzman2008empirical} for interesting work on FDR analysis for DTI
datasets.)  These voxels are where the large differences between Parkinson's disease group and control group
are observed. The results are shown in Figure~\ref{fig:PD_diff}. For a better visualization, we threshold the estimates by the top $1\%$. Note that, to
take into account the spatial structure, we apply a $4 \times 4 \times 4$
average mask to smooth the result. This smoothing may also be achieved by incorporating spatial regularization term in the expression for SURE~\eqref{eqn:sure}. However, the ensuing analysis becomes much more complicated and will be addressed in our future work. 

From the results, we can see that the
shrinkage effect of our Tweedie-adjusted estimate successfully corrects the
selection bias and produces more accurate identification of the affected
regions. Our method is able to capture the Substantia Nigra, which is the region
known to be affected by Parkinson's disease. \emph{Notably, our method did not
point to the apparently spurious and  isolated regions selected by the MOM
estimator (the tiny red spots in Figure~\ref{fig:mom})}.  We also mention that
the past research using FA-based analysis did not report the Internal Capsule as
a region affected by Parkinson's disease. We suspect that this discrepancy is
due to the fact that FA discards the directional information of the diffusion
tensors while we use the full diffusion tensor which contains the directional
information. We plan to conduct a large-scale experiment in our future work to
see if this observation continues to hold.

\begin{figure}[ht!]
\centering
\subfigure[MOM estimates]
{\includegraphics[width=0.465\textwidth]{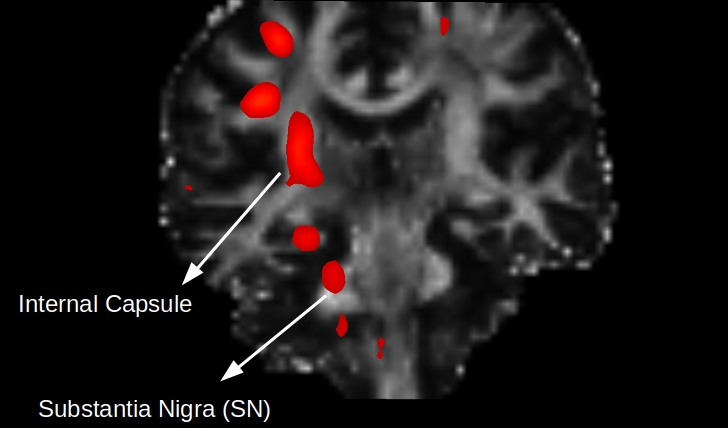} \label{fig:mom}}
\hfill
\subfigure[Tweedie-adjusted estimates]
{\includegraphics[width=0.465\textwidth]{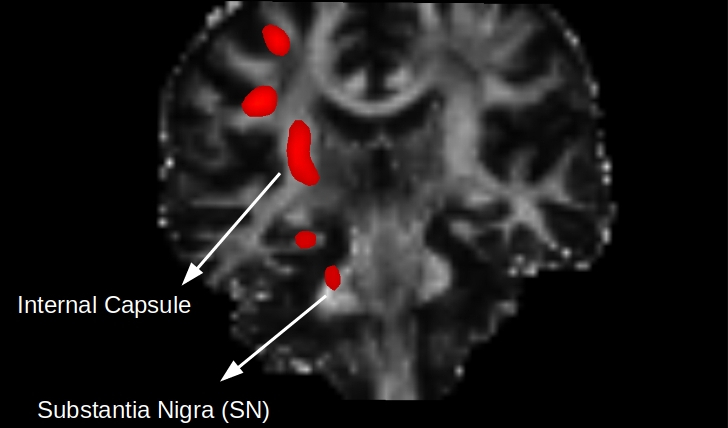} \label{fig:tweedie}}
\caption{Differences between Parkinson's disease and control groups are superimposed on dMRI
scans of a randomly-chosen Parkinson's disease patient and indicated in red.}
\label{fig:PD_diff}
\end{figure}

\section{Discussion and Conclusions}\label{conc}

In this work, we presented shrinkage estimators for the mean and covariance
of the Log-Normal distribution defined on the manifold $P_N$ of $N \times N$ SPD
matrices. We also showed that the proposed shrinkage estimators are
asymptotically optimal in a large class of estimators including the MLE. The
proposed shrinkage estimators are in closed form and resemble (in form) the James-Stein
estimator in Euclidean space $\mathbb{R}^p$. We demonstrated that the proposed
shrinkage estimators outperform the MLE via several synthetic data examples and
real data experiments using diffusion tensor MRI datasets. The improvements of the
proposed shrinkage estimators are significant especially in the small sample
size scenarios, which is very pertinent to medical imaging applications.

Our work reported here is however based on the Log-Euclidean metric, and one of
the drawbacks of this metric is that it is not affine (GL) invariant, which is a
desired property in some applications. Unfortunately, the derivation of the
shrinkage estimators under the GL-invariant metric is challenging due to the
fact that there is no closed-form expression for some elementary quantities such as
the sample FM which makes it almost impossible to derive the corresponding SURE.
Our future research efforts will focus on developing a general framework for
designing shrinkage estimators that are applicable to general Riemannian
manifolds. 

For applications in localizing the regions of the brain where the two groups
differ, our approach already works well, but it can potentially be improved if we take
into account the fact that neighboring voxels within a region are close under some 
measure of similarity. For instance, $M^{(k)}_i$ and $M^{(k)}_{j}$ should be close if 
voxels $i$ and $j$ are close. Currently, our approach is to apply a spatial smoother to the Tweedie-adjusted estimates. Instead, the improvement can be achieved by imposing regularization 
constraints, e.g.\ a spatial process prior, in the proposed framework. However, the ensuing analysis becomes rather 
complicated and will be the focus of our future efforts.\\


\newpage
\setcounter{section}{0}
\begin{center}
    {\LARGE\bf Supplement to ``An Empirical Bayes Approach to Shrinkage Estimation on the Manifold of Symmetric Positive-Definite Matrices''}
\end{center}

\bigskip

\begin{abstract}
In this supplement, we provide the technical details in section 3.2 of the main paper, including the derivation of the risk function and the SURE, the proofs for Theorem 5 and 6, and the implementation details for minimizing SURE. The notation used in this supplement is the same as in the main paper.
\end{abstract}

\section{Preliminary}%
\label{sec:preliminary}

Before presenting the proofs, we review the following elementary results for
both multivariate normal distributions and Wishart distributions which are use
extensively in the proofs. Let $X\sim N_p(\mu, \Sigma)$. Then
\begin{align*}
    E\|X - c\|^2 & = \tr \Sigma + \|\mu-c\|^2\\
    E\|X-c\|^4 & = (\tr\Sigma)^{2} + 2\tr(\Sigma^{2}) +
    4(\mu-c)^T\Sigma(\mu-c) + 2\|\mu-c\|^{2}\tr\Sigma + \|\mu-c\|^4
\end{align*}
where $c \in \mathbb{R}^p$. For $X\sim \text{LN}(M, \Sigma)$, 
\begin{align*}
    E d^2_{\text{LE}}(X, C) & = \tr \Sigma + d^2_{\text{LE}}(M, C)\\
    E d^4_{\text{LE}}(X, C) & = (\tr\Sigma)^{2} + 2\tr(\Sigma^{2}) +
    4(\widetilde{M}-\widetilde{C})^T\Sigma(\widetilde{M}-\widetilde{C}) +
    2d^2_{\text{LE}}(M, C)\tr\Sigma + d^4_{\text{LE}}(M, C)
\end{align*}
where $C \in P_N$. These results can be easily obtained from the definition of
the Log-Normal distribution.

Let $Y$ be a $p\times p$ symmetric matrix with eigenvalues
$\lambda_1,\ldots,\lambda_p$ and let $\kappa=(k_1,\ldots,k_p)$ be a (non-increasing) partition of
a positive integer $k$, i.e.\ $k_1\geq k_2\geq,\ldots, \geq k_p$ and
$\sum_{i=1}^pk_i=k$ where the $k_i$'s are non-negative integers. The \emph{zonal
polynomial} of $Y$ corresponding to $\kappa$, denoted by $C_\kappa(Y)$, is a
symmetric, homogeneous polynomial of degree $k$ in the eigenvalues
$\lambda_1,\ldots,\lambda_p$. One of the properties of the zonal polynomials is
that $\big(\tr Y \big)^k = \sum_{\kappa}C_\kappa(Y)$. For a more precise
definition of the zonal polynomial and its calculations, we refer the readers to
Ch.~7 of \citet{muirhead1982aspects}. The following lemma is essential in our
work.
\begin{lemma}\label{lem:zonal}(\citet{muirhead1982aspects}, Corollary 7.2.4) If
    $Y$ is positive definite, then $C_\kappa(Y)>0$ for all partitions $\kappa$.
\end{lemma}
The $j$th \emph{elementary symmetric function} of $Y$, denoted by $\tr_jY$, is
the sum of all principal minors of order $j$ of the matrix $Y$. For the case of
$j = 1$, $\tr_1 Y = \sum_{i=1}^p\lambda_i = \tr Y$, and for the case of $j = 2$,
$\tr_2 Y = \sum_{i < j}\lambda_i\lambda_j$. The definition gives rise to the following
identities which are useful in the proofs
\begin{align}
    (\tr Y)^2 & = \tr (Y^2) + 2 \tr_2 Y \label{eq:trace}\\
    (\tr Y)^4 & = (\tr(Y^2))^2 + 4\tr(Y^2)\tr_2Y + 4(\tr_2
    Y)^2 \nonumber \\
                      & = \tr(Y^4) + 2\tr_2Y^2 + 2 (\tr_2Y)(\tr
                      Y)^2\label{eq:trace2} 
\end{align}

For $S\sim\text{Wishart}_p(\Sigma, \nu)$, the $k$th moment, $k=0,1,2,\ldots$ of
$\tr S$ is given by
\begin{align}\label{eq:moment_tr}
    E(\tr S)^k = 2^k\sum_{\kappa}\Big(\frac{\nu}{2}\Big)_\kappa C_\kappa(\Sigma)
\end{align}
where 
\[
    (a)_\kappa = \prod_{i=1}^p\Big(a - \frac{i-1}{2}\Big)_{k_i}
\]
is called the \emph{generalized hypergeometric coefficient} and
\[
    (a)_{k_i}=a(a+1)\ldots(a+k_i-1),\,(a)_0=1
\]
(\citet{gupta2000matrix}, Theorem
3.3.23). Next, we review some elementary results for the Wishart distribution:
\begin{align*}
    ES & = \nu\Sigma\\
    ES^2 & = \nu(\nu+1)\Sigma^2 + \nu(\tr \Sigma)\Sigma\\
    E \tr_k S & = \nu(\nu-1)\cdots(\nu-k+1)\tr_k \Sigma\\
    E (\tr S)^2 & = \nu(\nu+2)(\tr \Sigma)^2 - 4\nu\tr_2\Sigma
                  = \nu^2(\tr\Sigma)^2 + 2\nu\tr(\Sigma^2).
\end{align*}
These results can be found in \citet{gupta2000matrix}(p.~99 and p.~106). 

\section{Derivation of the Risk Function and the SURE}
\label{sec:derivation_risk}

Recall that the loss function is $L\big((\widehat{\bd{M}}, \widehat{\bd{\Sigma}}), (\bd{M},\bd{\Sigma})\big) = p^{-1}\sum_{i=1}^pd_{\text{LE}}^2(\widehat{M}_i,M_i) + p^{-1}\sum_{i=1}^p\|\widehat{\Sigma}_i-\Sigma_i\|^2 = L_1(\widehat{\bd{M}}, \bd{M}) + L_2(\widehat{\bd{\Sigma}}, \bd{\Sigma})$. Write $R\big((\widehat{\bd{M}},\widehat{\bd{\Sigma}}),(\bd{M},\bd{\Sigma})\big) = EL_1(\widehat{\bd{M}}, \bd{M}) + EL_2(\widehat{\bd{\Sigma}}, \bd{\Sigma}) = R_1(\widehat{\bd{M}}, \bd{M}) + R_2(\widehat{\bd{\Sigma}}, \bd{\Sigma})$. Then 
\begin{align*}
    R_1(\widehat{\bd{M}}, \bd{M}) & = p^{-1}\sum_{i=1}^p Ed^2_{\text{LE}}(\widehat{M}_i,M_i) \\
    & = p^{-1}\sum_{i=1}^p\Big[ \frac{n^2}{(\lambda+n)^2}Ed^2_{\text{LE}}(\bar{X}_i, M_i) + \frac{\lambda^2}{(\lambda+n)^2}d^2_{\text{LE}}(\mu, M_i)\Big]\\
    & = p^{-1}(\lambda+n)^{-2}\sum_{i=1}^p\Big[n\text{tr}\Sigma_{i}+\lambda^{2}d^2_{\text{LE}}(\mu, M_{i})\Big]
\end{align*}
and 
\begin{align*}
    R_2(\widehat{\bd{\Sigma}}, \bd{\Sigma}) & = p^{-1}\sum_{i=1}^p(\nu+n-q-2)^{-2}E\|(\Psi-(\nu-q-1)\Sigma_i)+(S_i-(n-1)\Sigma_i)\|^2\\
    & = p^{-1}\sum_{i=1}^p(\nu+n-q-2)^{-2}\Big[ E\tr(S_i^2) - 2(n-1)E\tr(S_i\Sigma_i) + (n-1)^2\tr(\Sigma_i^2)\\
    & \qquad + \tr(\Psi^2) - 2(\nu-q-1)\tr(\Psi\Sigma_i) + (\nu-q-1)^2\tr(\Sigma_i^2)\Big]\\
    & = p^{-1}\sum_{i=1}^p(\nu+n-q-2)^{-2}\Big[ n(n-1)\tr(\Sigma_i^2) + (n-1)(\tr\Sigma_i)^2 - 2(n-1)^2\tr(\Sigma_i^2)\\
    & \qquad + (n-1)^2\tr(\Sigma_i^2) + \tr(\Psi^2) - 2(\nu-q-1)\tr(\Psi\Sigma_i) + (\nu-q-1)^2\tr(\Sigma_i^2)\Big]\\
    & = p^{-1}\sum_{i=1}^p(\nu+n-q-2)^{-2}\Big[\big(n-1+(\nu-q-1)^{2}\big)\text{tr}(\Sigma_{i}^{2})\\
    & \qquad - 2(\nu-q-1)\text{tr}(\Psi\Sigma_{i})+(n-1)(\text{tr}\Sigma_{i})^{2}+\text{tr}(\Psi^{2})\Big].
\end{align*}
To obtain the SURE for this risk function, we first have to find unbiased estimates of the quantities $\tr\Sigma_i$, $d^2_{\text{LE}}(\mu, M_i)$, $\tr(\Sigma_i^2)$, $\tr(\Psi\Sigma_i)$, and $(\tr\Sigma_i)^2$. With the results provided in Section~\ref{sec:preliminary}, it is easy to verify the following equations:
\begin{align*}
    \tr \Sigma_i & = E\big((n-1)^{-1}\tr S_i\big)\\
    \tr (\Psi \Sigma_i) & = E\big( (n-1)^{-1}\tr (\Psi S_i)\big)\\
    (\tr \Sigma_i)^2 & = E\Big(\frac{n(\tr S_i)^2 - 2\tr S_i^2}{(n-1)(n+1)(n-2)}\Big)\\
    \tr \Sigma_i^2 & = E\Big(\frac{(n-1)\tr S_i^2 - (\tr S_i)^2}{(n-1)(n+1)(n-2)}\Big)\\
    d^2_{\text{LE}}(\mu, M_i) & = E\Big( d^2_{\text{LE}}(\bar{X}_i, \mu) - \frac{\tr S_i}{n(n-1)}\Big).
\end{align*}
Plugging the above unbiased estimates into the risk function we obtain
\begin{align*}
\text{SURE}(\lambda,\Psi,\nu, \mu) & = p^{-1}\sum_{i=1}^{p}\Bigg\{(\lambda+n)^{-2}\Big[\frac{n}{n-1}\tr S_i + \lambda^{2}d^2_{\text{LE}}(\bar{X}_{i}, \mu) - \frac{\lambda^2}{n(n-1)}\tr S_i\Big]\\
 & \qquad + (\nu+n-q-2)^{-2}\Big[(n-1+(\nu-q-1)^{2})\Big(\frac{(n-1)\tr S_i^2 - (\tr S_i)^2}{(n-1)(n+1)(n-2)}\Big)\\
 & \qquad + \frac{n(\tr S_i)^2 - 2\tr S_i^2}{(n+1)(n-2)} - 2\frac{\nu-q-1}{n-1}\tr(\Psi S_i) + \tr(\Psi^2)\Big]\Bigg\}\\
 & = p^{-1}\Bigg\{\sum_{i=1}^{p}(\lambda+n)^{-2}\Big[\frac{n-\lambda^2/n}{n-1}\tr S_{i}+ \lambda^2d^2_{\text{LE}}(\bar{X}_i,\mu)\Big]\\
 & \qquad + (\nu+n-q-2)^{-2}\Big[\frac{n-3 + (\nu-q-1)^2}{(n+1)(n-2)}\tr(S^2_i)\\
 & \qquad + \frac{(n-1)^2-(\nu-q-1)^2}{(n-1)(n+1)(n-2)}\big(\tr S_i\big)^2 - 2\frac{\nu-q-1}{n-1}\tr(\Psi S_i) + \tr (\Psi^2)\Big]\Bigg\}.
\end{align*}

\section{Proofs of the Theorems}%
\label{sec:proofs_of_the_theorems}

The following lemmas are essential for proving Theorem~\ref{thm3}.
\begin{lemma}\label{lem:xie}\citep{xie2012sure} Let $X_i\ind N(\theta_i,A_i)$. Assume the
    following conditions:
    \begin{enumerate}
        \item[(i)] $\limsup_{p\to\infty} p^{-1}\sum_{i=1}^p A_i^2 < \infty$,
        \item[(ii)] $\limsup_{p\to\infty} p^{-1}\sum_{i=1}^p A_i\theta_i^2 < \infty$,
        \item[(iii)] $\limsup_{p\to\infty} p^{-1}\sum_{i=1}^p |\theta_i|^{2+\delta} < \infty$ for some $\delta > 0$.
    \end{enumerate}
    Then $E\big(\max_{1\leq i \leq p} X_i^2\big) = O(p^{2/(2+\delta^*)})$ where
    $\delta^* = \min(1, \delta)$.
\end{lemma}
The next lemma is an extension of the previous lemma to Log-Normal
distributions.
\begin{lemma}\label{lem:LN} Let $X_i \ind LN(M_i, \Sigma_i)$ on $P_N$ and $q = N(N+1)/2$.
    Assume the following conditions:
    \begin{enumerate}
        \item[(i)] $\limsup_{p\to\infty} p^{-1}\sum_{i=1}^p \big(\tr
            \Sigma_i\big)^2 < \infty$,
        \item[(ii)] $\limsup_{p\to\infty} p^{-1}\sum_{i=1}^p
            \widetilde{M}^T_i\Sigma_i\widetilde{M}_i < \infty$,
        \item[(iii)] $\limsup_{p\to\infty} p^{-1}\sum_{i=1}^p \Vert \! \log M_i
            \Vert^{2+\delta} < \infty$ for some $\delta > 0$.
    \end{enumerate}
    Then $E\big(\max_{1\leq i \leq p} \|\log X_i \|^2\big) =
    O(p^{2/(2+\delta^*)})$ where $\delta^*=\min(1, \delta)$. 
\end{lemma}
\begin{proof}
    Write $Y_i = \widetilde{X}_i$ and $\mu_i = \widetilde{M}_i$. Then $\|Y_i\|^2
    = \|\log X_i\|^2$. From the definition of the Log-Normal distribution, $Y_i
    \ind N_q(\mu_i, \Sigma_i)$. Since for $j = 1,\ldots,q$
    \begin{align*}
        \sum_{i=1}^p \Sigma_{i,jj}^2 & < \sum_{i=1}^p\big(\tr \Sigma_i\big)^2\\ 
        \sum_{i=1}^p \Sigma_{i,jj}\mu_{i,j}^2 & < \sum_{i=1}^p
        \mu_i^T\Sigma_i\mu_i = \sum_{i=1}^p \widetilde{M}^T_i\Sigma_i\widetilde{M}_i\\
        \sum_{i=1}^p |\mu_{i,j}|^{2+\delta} & < \sum_{i=1}^p \|\mu_i\|^{2+\delta} = 
        \sum_{i=1}^p \Vert \! \log M_i \Vert^{2+\delta},
    \end{align*}
    by Lemma~\ref{lem:xie}, we have $E(\max_{1\leq i \leq, p}Y^2_{i,j}) =
    O(p^{2/(2+\delta^*)})$. Then 
    \begin{align*}
        E\Big(\max_{1\leq i \leq p} \|\log X_i \|^2\Big) & = E\Big(\max_{1\leq i
        \leq p} \|Y_i\|^2\Big)\\
        & \leq E \Big( \sum_{j=1}^q \max_{1\leq i \leq p} Y^2_{i,j}\Big)
        = \sum_{j=1}^qE \Big(  \max_{1\leq i \leq p} Y^2_{i,j}\Big)
        = O(p^{2/(2+\delta*)})
    \end{align*}
    which concludes the proof.
\end{proof}

\begin{lemma}\label{lem:S_bound}
    Let $S_i\ind\text{Wishart}(\Sigma_i, \nu)$ where the $\Sigma_i$'s are
    $q\times q$ symmetric positive-definite matrices. If
    $\limsup_{p\to\infty} p^{-1}\sum_{i=1}^p (\tr\Sigma_i)^4 < \infty$, 
    then $E(\max_{1 \leq i \leq p}\|S_i\|^2) = O(q^2p^{1/2}(\log p)^2 +
    q^2p^{1/2}(\log q)^2)$.
\end{lemma}
\begin{proof}
    Write $S_i = X_iX_i^T$ where $X_i\ind N_q(0, \Sigma_i)$. Then
    $\|S_i\|^2=\|X_i\|^4$. From Lemma~\ref{lem:xie}, we have $E(\max_{1\leq j
    \leq q}X_{i,j}^2) = O(p^{2/3})$ and $E(\max_{1\leq i \leq p, 1\leq j \leq
q}X_{i,j}^2) = O(q^{2/3}p^{2/3})$. Let $X_{i,j} = \Sigma_{i,j}^{1/2}Z_{i,j}$
where $Z_{i,j}\iid N(0,1)$. Then $X_{i,j}^4=\Sigma_{i,jj}^2Z_{i,j}$ and
\[
    \max_{1 \leq i \leq p, 1\leq j\leq q}X_{i,j}^4 \leq \max_{1 \leq i \leq p, q \leq
j \leq q}\Sigma_{i,jj}^2\cdot\max_{ 1\leq i \leq p, 1 \leq j \leq q}Z_{i,j}.
\]
Since
\begin{align*}
    \max_{1\leq i \leq p, 1 \leq j \leq q}\Sigma^4_{i,jj} & < \max_{1 \leq i \leq
    p} \tr(\Sigma_i)^4 < \sum_{i=1}^p\tr(\Sigma_i)^4 = O(p)
\end{align*}
implies $\max_{1\leq i \leq p, 1 \leq j \leq q}\Sigma_{i,jj}^2 = O(p^{1/2})$ and
\begin{align*}
    E\Big(\max_{1\leq i \leq p, 1 \leq j \leq q}Z_{i,j}^4\Big) & = O((\log p +
    \log q)^2), 
\end{align*}
we have
\begin{align*}
    E\Big(\max_{1\leq i \leq p, 1 \leq j \leq q}X_{i,j}^4\Big) & \leq \max_{1 \leq
        i \leq p, q \leq j \leq q}\Sigma_{i,jj}^2 E\Big(\max_{ 1\leq i \leq p, 1 \leq
    j \leq q}Z_{i,j}\Big) = O(p^{1/2}(\log p + \log q)^2).
\end{align*}

Then
    \begin{align*}
        E\Big(\max_{1 \leq i \leq p}\|S_i\|^2\Big) & = E\Big(\max_{1 \leq i \leq
        p}\|X_i\|^4\Big) = E\Bigg[\max_{1 \leq i \leq
p}\Bigg(\sum_{j=1}^qX_{i,j}^2\Bigg)^2\Bigg]\\
        & = E\Bigg[\max_{1\leq i \leq p}\Bigg(\sum_{j=1}^qX_{i,j}^4 +
        \sum_{j\neq k}X_{i,j}^2X_{i,k}^2\Bigg)\Bigg]\\
        & \leq qE\Big(\max_{1\leq i \leq p, 1\leq j \leq q}
        X_{i,j}^4\Big)+q(q-1)E\Big(\max_{1 \leq i \leq p, 1 \leq j \leq
    q}X_{i,j}^4\Big)\\
        & \leq q^2O(p^{1/2}(\log p + \log q)^2)\\
        & = O(q^2p^{1/2}(\log p)^2 + q^2p^{1/2}(\log q)^2)
    \end{align*}

\end{proof}

\setcounter{theorem}{4}
\begin{theorem}\label{thm3}
    Assume the following conditions:
\begin{enumerate}
    \item[(i)] $\limsup_{p\to\infty} p^{-1}\sum_{i=1}^p \big(\tr \Sigma_i\big)^4 < \infty$,
    \item[(ii)] $\limsup_{p\to\infty} p^{-1}\sum_{i=1}^p
        \widetilde{M}^T_i\Sigma_i\widetilde{M}_i < \infty$,
    \item[(iii)] $\limsup_{p\to\infty} p^{-1}\sum_{i=1}^p \Vert \! \log M_i \Vert^{2+\delta} < \infty$ for some $\delta > 0$.
\end{enumerate}
Then
\[
    \sup_{\substack{\lambda > 0, \nu > q+1, \|\Psi\|\leq \max_{1 \leq i \leq
                p}\|S_i\|,\\ \|\log\mu\|\leq \max_{1\leq i\leq
    p}\|\log\bar{X}_i\|}} \Big|\text{SURE}(\lambda, \Psi, \nu, \mu)-
    L\Big(\big(\widehat{\bd{M}}^{\lambda,\mu},
            \widehat{\bd{\Sigma}}^{\Psi, \nu}\big), \big(\bd{M},
    \bd{\Sigma}\big)\Big)\Big| \stackrel{\text{prob}}{\longrightarrow} 0\quad
    \text{as }p \to \infty.
\]
\end{theorem}

\begin{proof}
    First, we write the loss function $L$ as
    \begin{align*}
     L\Big(\big(\widehat{\bd{M}}^{\lambda,\mu},
            \widehat{\bd{\Sigma}}^{\Psi, \nu}\big), \big(\bd{M},
    \bd{\Sigma}\big)\Big) & = p^{-1}\sum_{i=1}^p
    d^2_{\text{LE}}(\widehat{M}^{\lambda,\mu}_i, M_i) +
    \|\widehat{\Sigma}^{\Psi,\nu}_i-\Sigma_i\|^2\\
    & = L_1\Big(\widehat{\bd{M}}^{\lambda,\mu}, \bd{M}\Big) +
    L_2\Big(\widehat{\bd{\Sigma}}^{\Psi,\nu}, \bd{\Sigma}\Big),\\
    \end{align*}
    where
    \begin{align*}
        L_1\Big(\widehat{\bd{M}}^{\lambda,\mu}, \bd{M}\Big) & = p^{-1}\sum_{i=1}^p\Big\| (\lambda+n)^{-1}\Big(n\Big(\log
    \bar{X}_i - \log M_i\Big) + \lambda\Big(\log \mu - \log
    M_i\Big)\Big)\Big\|^2 \\
    & = p^{-1}\sum_{i=1}^p(\lambda+n)^{-2}\Big[n^2d^2_{\text{LE}}(\bar{X}_i,
    M_i) + \lambda^2d^2_{\text{LE}}(\mu, M_i)\\
    & \quad + 2n\lambda\big\langle\log
\bar{X}_i-\log M_i, \log \mu - \log M_i \big\rangle\Big],\\
L_2\Big(\widehat{\bd{\Sigma}}^{\Psi,\nu},\bd{\Sigma}\Big) & = p^{-1}\sum_{i=1}^p
(\nu+n-q-2)^{-2}\Big\|\big(\Psi-(\nu-q-1)\Sigma_i\big) +
\big(S_i-(n-1)\Sigma_i\big)\Big\|^2\\
    & = p^{-1}\sum_{i=1}^p (\nu+n-q-2)^{-2}\Big[\tr\big(\Psi^2\big) -
    2(\nu-q-1)\tr\big(\Psi\Sigma_i\big) + (\nu-q-1)^2\tr\big(\Sigma_i^2\big)\\
    & \quad + \tr\big(S_i^2\big) - 2(n-1)\tr\big(S_i\Sigma_i\big) +
(n-1)^2\tr\big(\Sigma_i^2\big)\\
    & \quad + 2\langle \Psi-(\nu-q-1)\Sigma_i,S_i-(n-1)\Sigma_i\rangle\Big].
    \end{align*}
Write the SURE as 
\begin{align*}
    \text{SURE}(\lambda, \mu, \Psi, \nu) & = \text{SURE}_1(\lambda, \mu) + \text{SURE}_2(\Psi, \nu),            
\end{align*}
where 
\begin{align*}
    \text{SURE}_1(\lambda,\mu) & =
    p^{-1}\sum_{i=1}^p(\lambda+n)^{-2}\Big[\frac{n-\lambda^2/n}{n-1}\tr
    S_i + \lambda^2d^2_{\text{LE}}(\bar{X}_i, \mu)\Big],\\
    \text{SURE}_2(\Psi, \nu) & = p^{-1}\sum_{i=1}^p (\nu + n - q -
    2)^{-2}\Bigg[\frac{n-3+(\nu-q-1)^2}{(n+1)(n-2)}\tr\big(S_i^2\big)\\
    & \quad + \frac{(n-1)^2-(\nu-q-1)^2}{(n-2)(n-1)(n+1)}\big(\tr
S_i\big)^2-2\frac{\nu-q-1}{n-1}\tr\big(\Psi S_i\big)
+ \tr \big(\Psi^2\big)\Bigg].
\end{align*}
Since 
\begin{multline*}
    \sup_{\substack{\lambda > 0, \nu > q+1, \|\Psi\|\leq \max_{1\leq i \leq
                p}\|S_i\|,\\ \|\log\mu\|\leq \max_{1\leq i\leq
    p}\|\log\bar{X}_i\|}} \Big|\text{SURE}(\lambda, \Psi, \nu, \mu)-
    L\Big(\big(\widehat{\bd{M}}^{\lambda,\mu},
            \widehat{\bd{\Sigma}}^{\Psi, \nu}\big), \big(\bd{M},
    \bd{\Sigma}\big)\Big)\Big| \leq \\
    \sup_{\lambda>0, \|\log\mu\|\leq \max_{1\leq i \leq p}\|\log \bar{X}_i
    \|}\big|\text{SURE}_1(\lambda,\mu) - L_1\big(\widehat{\bd{M}}^{\lambda,\mu},
\bd{M}\big)\big| \\
+ \sup_{\nu>q+1, \|\Psi\|\leq \max_{1\leq i \leq p}\|S_i\|} \big|\text{SURE}_2(\Psi,\nu) -
    L_2\big(\widehat{\bd{\Sigma}}^{\Psi,\nu}, \bd{\Sigma}\big)\big|, 
\end{multline*}
it suffices to show the two terms on the right-hand side converge to 0 in probability. For
the first term, 
\begin{align}
    |\text{SURE}_1(\lambda, \mu) -
    L_1\big(\widehat{\bd{M}}^{\lambda,\mu},\bd{M}\big)| & = \Bigg|
    p^{-1}\sum_{i=1}^p(\lambda+n)^{-2}\Bigg[\frac{n}{n-1}\tr S_i -
    n^2d^2_{\text{LE}}(\bar{X}_i, M_i) \nonumber \\
    & \quad + \lambda^2\Big(\frac{\tr S_i}{n(n-1)} +
    d^2_{\text{LE}}(\bar{X}_i, \mu)-d^2_{\text{LE}}(M_i, \mu)\Big) \nonumber \\
    & \quad + 2n\lambda\Big\langle \log \bar{X}_i-\log M_i, \log \mu - \log
    M_i\Big\rangle\Bigg]\Bigg| \nonumber \\ 
    & \leq \Bigg| p^{-1}\sum_{i=1}^p(\lambda+n)^{-2}\Big(\frac{n}{n-1}\tr S_i -
    n^2d^2_{\text{LE}}(\bar{X}_i, M_i) \Big)\Bigg| \label{eq:M_1} \\
    & \quad + \Bigg| p^{-1}\sum_{i=1}^p\frac{\lambda^2}{(\lambda+n)^2}
    \Big(\frac{\tr S_i}{n(n-1)} + d^2_{\text{LE}}(\bar{X}_i,
    \mu)-d^2_{\text{LE}}(M_i, \mu) \Big)\Bigg| \label{eq:M_2} \\
    & \quad + \Bigg|
    p^{-1}\sum_{i=1}^p\frac{2n\lambda}{(\lambda+n)^2}\Big\langle \log
    \bar{X}_i-\log M_i, \log \mu - \log M_i\Big\rangle \Bigg|. \label{eq:M_3}
\end{align}
We will now prove the convergence of each of the three terms individually.

For \eqref{eq:M_1}, from assumption (i), we have 
\begin{align*}
    & \quad \var \Bigg( p^{-1}\sum_{i=1}^p\Big(\frac{n}{n-1}\tr S_i -
    n^2d^2_{\text{LE}}(\bar{X}_i, M_i) \Big)\Bigg)\\
    & = \frac{1}{p} p^{-1}\sum_{i=1}^p \var \Big(\frac{n}{n-1}\tr S_i -
    n^2d^2_{\text{LE}}(\bar{X}_i, M_i) \Big)\\
    & = \frac{1}{p}p^{-1}\sum_{i=1}^pE\Big(\frac{n}{n-1}\tr S_i -
    n^2d^2_{\text{LE}}(\bar{X}_i,M_i)\Big)^2\\
    & = \frac{n^2}{p}p^{-1}\sum_{i=1}^p \Big[ \frac{E\big(\tr
            S_i\big)^2}{(n-1)^2} + n^2Ed^4_{\text{LE}}(\bar{X}_i,M_i) -
            2\frac{n}{n-1}E\big(\tr
    S_i\big)Ed^2_{\text{LE}}(\bar{X}_i, M_i)\Big]\\  
    & = \frac{n^2}{p}p^{-1}\sum_{i=1}^p \Big[ \frac{n+1}{n-1}\big(\tr
        \Sigma_i\big)^2 + 4\frac{\tr_2\Sigma_i}{n-1} +
            \big(\tr\Sigma_i\big)^2 + 2\tr\big(\Sigma_i^2\big) -
            2\big(\tr\Sigma_i\big)^2\Big]\\  
    & = \frac{n^2}{p}p^{-1}\sum_{i=1}^p \Big[ \frac{2}{n-1}\big(\tr
        \Sigma_i\big)^2 + \frac{4}{n-1}\tr_2\Sigma_i +
        2\tr\big(\Sigma_i^2\big)\Big] \stackrel{p \to \infty}{\longrightarrow}
        0.  
\end{align*}
Then, by Markov's inequality,
\[
\Bigg| p^{-1}\sum_{i=1}^p\Big(\frac{n}{n-1}\tr S_i - n^2d^2_{\text{LE}}(\bar{X}_i, M_i) \Big)\Bigg| \stackrel{\text{prob}}{\longrightarrow} 0 \quad \text{as }p \to \infty.
\]
Thus
\begin{align}\label{eq:M_1_pf}
    & \quad \sup_{\lambda>0}\Bigg|p^{-1}\sum_{i=1}^p(\lambda+n)^{-2}\Big(\frac{n}{n-1}\tr S_i -
    n^2d^2_{\text{LE}}(\bar{X}_i, M_i)\Big)\Bigg| \nonumber \\
    & = \Big( \sup_{\lambda>0}(\lambda+n)^{-2}\Big) \Bigg| p^{-1}\sum_{i=1}^p\Big(\frac{n}{n-1}\tr S_i - n^2d^2_{\text{LE}}(\bar{X}_i, M_i) \Big)\Bigg| \nonumber \\
    & = \frac{1}{n^2} \Bigg| p^{-1}\sum_{i=1}^p\Big(\frac{n}{n-1}\tr S_i - n^2d^2_{\text{LE}}(\bar{X}_i, M_i) \Big)\Bigg|
    \stackrel{\text{prob}}{\longrightarrow} 0 \quad \text{as }p \to \infty. 
\end{align}
\begin{remark}
    By the identity~\eqref{eq:trace}, assumption (i) implies
    $\limsup_{p\to\infty}p^{-1}\sum_{i=1}^p \tr \big(\Sigma_i^2\big) < \infty$
    and $\limsup_{p\to\infty}p^{-1}\sum_{i=1}^p \tr_2\Sigma_i < \infty$.
\end{remark}

For~\eqref{eq:M_2},
\begin{align*}
    & \quad \sup_{\lambda>0, \|\log\mu\|\leq \max_{1\leq i\leq p}\|\log\bar{X}_i\|}\Bigg| p^{-1}\sum_{i=1}^p\frac{\lambda^2}{(\lambda+n)^2}
    \Big(\frac{\tr S_i}{n(n-1)} + d^2_{\text{LE}}(\bar{X}_i,
    \mu)-d^2_{\text{LE}}(M_i, \mu) \Big)\Bigg|\\
    & = \sup_{\lambda>0, \|\log\mu\|\leq \max_{1\leq i\leq p}\|\log\bar{X}_i\|}\Bigg| p^{-1}\sum_{i=1}^p\frac{\lambda^2}{(\lambda+n)^2}
    \Big(\frac{\tr S_i}{n(n-1)} + \|\log\bar{X}_i\|^2 - \|\log M_i\|^2\\
    & \qquad + 2\langle \log \bar{X}_i - \log M_i, \log \mu\rangle
\Big)\Bigg|\\
    & \leq \sup_{\lambda>0}\Bigg| p^{-1}\sum_{i=1}^p\frac{\lambda^2}{(\lambda+n)^2}
    \Big(\frac{\tr S_i}{n(n-1)} + \|\log\bar{X}_i\|^2 - \|\log M_i\|^2\Big)\Bigg|\\
    & \qquad + \sup_{\lambda>0, \|\log\mu\|\leq \max_{1\leq i\leq p}\|\log\bar{X}_i\|}
    \Bigg| p^{-1}\frac{2\lambda^2}{(\lambda+n)^2}
    \Big\langle \sum_{i=1}^p \log \bar{X}_i - \log M_i, \log \mu\Big\rangle\Bigg|\\
    & \leq \Big(\sup_{\lambda>0}\frac{\lambda^2}{(\lambda+n)^2}\Big)\Bigg| p^{-1}\sum_{i=1}^p
    \Big(\frac{\tr S_i}{n(n-1)} + \|\log\bar{X}_i\|^2 - \|\log M_i\|^2\Big)\Bigg|\\
    & \qquad + \sup_{\lambda>0, \|\log\mu\|\leq \max_{1\leq i\leq p}\|\log\bar{X}_i\|}
    \Bigg| p^{-1}\frac{2\lambda^2}{(\lambda+n)^2}
    \|\!\log\mu\|
    \Bigg\| \sum_{i=1}^p \big(\log\bar{X}_i-\log M_i \big)\Bigg\|\Bigg|\\
    & \pushright{\text{(By Cauchy's inequality)}}\\
    & \leq \Bigg| p^{-1}\sum_{i=1}^p
    \Big(\frac{\tr S_i}{n(n-1)} + \|\log\bar{X}_i\|^2 - \|\log M_i\|^2\Big)\Bigg|\\
    & \qquad + \Bigg|p^{-1}\max_{1\leq i\leq p}\|\log\bar{X}_i\|
    \Bigg\| \sum_{i=1}^p \big(\log\bar{X}_i-\log M_i \big)\Bigg\|\Bigg|
\end{align*}
Since by assumptions (i) and (ii), 
\begin{align*}
    & \quad \var \Bigg( p^{-1}\sum_{i=1}^p
    \Big(\frac{\tr S_i}{n(n-1)} + \|\log\bar{X}_i\|^2 - \|\log M_i\|^2\Big)\Bigg)\\
    & = p^{-2}\sum_{i=1}^p\var \Big(\frac{\tr
    S_i}{n(n-1)} + \|\log\bar{X}_i\|^2 - \|\log M_i\|^2\Big)\\
    & = p^{-2}\sum_{i=1}^p E \Big(\frac{\tr
    S_i}{n(n-1)} + \|\log\bar{X}_i\|^2 - \|\log M_i\|^2\Big)^2\\
    & = p^{-2}\sum_{i=1}^p \Bigg[
        \Bigg(\frac{n+1}{n^2(n-1)}\big(\tr\Sigma_i\big)^2-\frac{4}{n^2(n-1)}\tr_2\Sigma_i\Bigg)
        \Bigg(+ \frac{\big(\tr\Sigma_i\big)^2}{n^2} +
            \frac{2\tr\big(\Sigma_i^2\big)}{n^2}\\
    & \qquad + \frac{4}{n}\widetilde{M}_i^T\Sigma_i\widetilde{M}_i +
\frac{2}{n}\|\log M_i\|^2\tr\Sigma_i + \|\log M_i \|^4\Bigg) + \|\log M_i \|^4\\ 
    & \qquad + \frac{2}{n}\tr\Sigma_i\Big(\frac{1}{n}\tr\Sigma_i + \|\log M_i
\|^2\Big) -2 \Big(\frac{1}{n}\tr\Sigma_i + \|\log M_i \|^2\Big)\|\log M_i\|^2 -
\frac{2}{n}\tr\Sigma_i\|\log M_i\|^2\Bigg]\\
    & = p^{-2}\sum_{i=1}^p \Bigg[
    \frac{4n-2}{n^2(n-1)}\big(\tr\Sigma_i\big)^2 -
\frac{4}{n^2(n-1)}\tr_2\Sigma_i + \frac{2}{n}\tr\big(\Sigma_i^2\big) +
\frac{4}{n}\widetilde{M}_i^T\Sigma_i\widetilde{M}_i \Bigg]
\stackrel{p\to\infty}{\longrightarrow} 0,
\end{align*}
we have 
\begin{align*}
    \Bigg| p^{-1}\sum_{i=1}^p \Big(\frac{\tr S_i}{n(n-1)} + \|\log\bar{X}_i\|^2 - \|\log
    M_i\|^2\Big)\Bigg| \stackrel{\text{prob}}{\longrightarrow} 0 \quad \text{as
}p \to \infty
\end{align*}
by Markov's inequality. Since by Lemma~\ref{lem:LN},
\begin{align*}
    & \quad E \Bigg[ \frac{2}{p} \max_{1\leq i\leq p}\|\log\bar{X}_i\| \Bigg\| \sum_{i=1}^p
    \big(\log\bar{X}_i-\log M_i \big)\Bigg\|\Bigg]\\
    & \leq \frac{2}{p}\Bigg[ E\Big(\max_{1\leq i \leq
    p}\|\log\bar{X}_i\|^2\Big)E \Bigg\| \sum_{i=1}^p\big(\log \bar{X}_i - \log
    M_i\big)\Bigg\|^2\Bigg]^{1/2}\\
    & = O(p^{-1})\times O(p^{1/(2+\delta^*)}) \times O(p^{1/2})\\
    & = O(p^{-\delta^*/(4+2\delta^*)}),
\end{align*}
we have 
\begin{multline}
    \sup_{\lambda>0, \|\log\mu\|\leq \max_{1 \leq i \leq p}\|\log\bar{X}_i\|}
    \Bigg| p^{-1}\sum_{i=1}^p\frac{\lambda^2}{(\lambda+n)^2} \Big(\frac{\tr
    S_i}{n(n-1)} + d^2_{\text{LE}}(\bar{X}_i, \mu)-d^2_{\text{LE}}(M_i, \mu)
\Big)\Bigg| \stackrel{\text{prob}}{\longrightarrow} 0\\ \quad \text{as
}p\to\infty.\label{eq:M_2_pf}
\end{multline}

For~\eqref{eq:M_3}, we have
\begin{align*}
    & \quad \sup_{\lambda>0,\|\log\mu\|\leq\max_{1\leq i\leq p}\|\log\bar{X}_i\|}\Bigg|p^{-1}\sum_{i=1}^p\frac{2n\lambda}{(\lambda+n)^2}\Big\langle \log
    \bar{X}_i-\log M_i, \log \mu - \log M_i\Big\rangle \Bigg|\\
    & \leq
    \sup_{\lambda>0}\Bigg|p^{-1}\frac{2n\lambda}{(\lambda+n)^2}\max_{1\leq i \leq p}\|\log\bar{X}_i\|\Bigg\|\sum_{i=1}^p\big(
\log\bar{X}_i-\log M_i\big)\Bigg\|\Bigg|\\
    & \qquad +
    \sup_{\lambda>0}\Bigg|p^{-1}\sum_{i=1}^p\frac{2n\lambda}{(\lambda+n)^2}\Big\langle
    \log \bar{X}_i-\log M_i, \log M_i\Big\rangle \Bigg|\\
    & = \Bigg|\frac{1}{2p}\max_{1\leq i \leq p}\|\log\bar{X}_i\|\Bigg\|\sum_{i=1}^p\big(
\log\bar{X}_i-\log M_i\big)\Bigg\|\Bigg| +
    \Bigg|\frac{1}{2p}\sum_{i=1}^p\Big\langle
    \log \bar{X}_i-\log M_i, \log M_i\Big\rangle \Bigg|
\end{align*}
since $\sup_{\lambda>0}2n\lambda/(\lambda+n)^2 = 1/2$.
By assumption (ii), we have
\begin{align*}
    & \quad \var\Bigg[ p^{-1}\sum_{i=1}^p\Big\langle
    \log \bar{X}_i-\log M_i, \log M_i\Big\rangle\Bigg]\\
    & = p^{-2}\sum_{i=1}^pE\Big\langle
    \log \bar{X}_i-\log M_i, \log M_i\Big\rangle^2\\
    & = p^{-2}\sum_{i=1}^pE\widetilde{M}_i^T
    \Big[ \Big(\widetilde{\bar{X}}_i-\widetilde{M}_i\Big)
    \Big(\widetilde{\bar{X}}_i-\widetilde{M}_i\Big)^T \Big]\widetilde{M}_i\\
    & = p^{-2}\sum_{i=1}^p\frac{1}{n}\widetilde{M}_i^T
    \Sigma_i\widetilde{M}_i \stackrel{p \to \infty}{\longrightarrow} 0
\end{align*}
and again by Markov's inequality,
\begin{align*}
    \Bigg|\frac{1}{2p}\sum_{i=1}^p\Big\langle
    \log \bar{X}_i-\log M_i, \log M_i\Big\rangle \Bigg|
    \stackrel{\text{prob}}{\longrightarrow} 0 \quad \text{as }p \to \infty.
\end{align*}
Thus, 
\begin{align}
    \sup_{\lambda>0,\|\log\mu\|\leq\max_{1\leq i\leq p}\|\log\bar{X}_i\|}\Bigg|p^{-1}\sum_{i=1}^p\frac{2n\lambda}{(\lambda+n)^2}\Big\langle \log
    \bar{X}_i-\log M_i, \log \mu - \log M_i\Big\rangle \Bigg| \stackrel{\text{prob}}{\longrightarrow} 0 \quad \text{as
}p\to\infty.\label{eq:M_3_pf}
\end{align}
Combining~\eqref{eq:M_1_pf}, \eqref{eq:M_2_pf}, and \eqref{eq:M_3_pf}, we have
\begin{align}
    \sup_{\lambda>0, \|\log\mu\|\leq \max_{1\leq i \leq p}\|\log \bar{X}_i
    \|}\big|\text{SURE}_1(\lambda,\mu) - L_1\big(\widehat{\bd{M}}^{\lambda,\mu},
\bd{M}\big)\big|  \stackrel{\text{prob}}{\longrightarrow} 0 \quad \text{as
}p\to\infty. \label{eq:M_pf}
\end{align}

For the second term, we have 
\begin{align}
    & \quad |\text{SURE}_2(\Psi, \nu) -
    L_2\big(\widehat{\bd{\Sigma}}^{\Psi,\nu},\bd{\Sigma}\big)|\nonumber\\
    & = \Bigg|p^{-1}\sum_{i=1}^p(\nu+n-q-2)^{-2}\Bigg[2(\nu-q-1)
        \Big(\frac{\tr(\Psi S_i)}{n-1}-\tr(\Psi\Sigma_i)\Big) \nonumber \\
    & \quad + \frac{(n-1)^2-(\nu-q-1)^2}{(n+1)(n-2)(n-1)}(\tr S_i)^2 -
    \frac{(n-1)^2-(\nu-q-1)^2}{(n+1)(n-2)}\tr(S_i^2) \nonumber\\ 
    & \quad +2(n-1)\tr(S_i\Sigma_i) - \big((n-1)^2+(\nu-q-1)^2\big)
    \tr(\Sigma_i^2) \nonumber\\ 
    & \quad - 2\big\langle \Psi-(\nu-q-1)\Sigma_i, S_i - (n-1)\Sigma_i
    \big\rangle\Bigg]\Bigg| \nonumber \\ 
    & \leq \Bigg|
    p^{-1}\sum_{i=1}^p\frac{2(\nu-q-1)(n-1)}{(\nu+n-q-2)^2}\langle\Psi,
    S_i-(n-1)\tr\Sigma_i \rangle\Bigg| \label{eq:S_1} \\
    & \quad + \Bigg| p^{-1}\sum_{i=1}^p C(\nu)\Big[(\tr
    S_i)^2-(n-1)^2(\tr\Sigma_i)^2 -
    2(n-1)\tr(\Sigma_i^2) \Big]\Bigg|\label{eq:S_2}\\
    & \quad + \Bigg| p^{-1}\sum_{i=1}^p (n-1)C(\nu)\Big[\tr
        (S_i^2) - (n-1)(\tr\Sigma_i)^2 -
    n(n-1)\tr(\Sigma_i^2) \Big]\Bigg|\label{eq:S_3}\\
    & \quad + \Bigg| p^{-1}\sum_{i=1}^p 2(n-1)\Big[\tr
        (S_i\Sigma_i) - (n-1)\tr(\Sigma_i^2) \Big]\Bigg|\label{eq:S_4}\\
    & \quad + \Bigg|
    p^{-1}\sum_{i=1}^p\frac{2\big\langle \Psi-(\nu-q-1)\Sigma_i, S_i - (n-1)\Sigma_i
    \big\rangle}{(\nu+n-q-2)^2}\Bigg|. \label{eq:S_5}
\end{align}
where 
\begin{align*}
    C(\nu) =  (\nu+n-q-2)^{-2}\frac{(n-1)^2-(\nu-q-1)^2}{(n+1)(n-2)(n-1)}.
\end{align*}
Note that $\sup_{\nu>q-1}C(\nu) = [(n+1)(n-2)(n-1)]^{-1}$.

For~\eqref{eq:S_1}, by Lemma~\ref{lem:S_bound}, we have
\begin{align*}
    & \quad \sup_{\nu > q+1,\|\Psi\|\leq \max_{1\leq i \leq p}\|S_i\|}\Bigg|
    p^{-1}\sum_{i=1}^p\frac{2(\nu-q-1)(n-1)}{(\nu+n-q-2)^2} \big\langle \Psi,
    S_i - (n-1)\Sigma_i \big\rangle\Bigg|\\ 
    & \leq \sup_{\nu>q+1}\Bigg(\frac{2(\nu-q-1)(n-1)}{(\nu+n-q-2)^2}p^{-1}\max_{1\leq i \leq
    p}\|S_i\|\Bigg\|\sum_{i=1}^pS_i-(n-1)\Sigma_i\Bigg\|\Bigg)\\
    & = \frac{2(n-1)^2}{(n-2)^2}\Bigg(p^{-1}\max_{1\leq i \leq
    p}\|S_i\|\Bigg\|\sum_{i=1}^pS_i-(n-1)\Sigma_i\Bigg\|\Bigg)
\end{align*}
and
\begin{align*}
    & \quad E\Bigg[ \frac{1}{p}\max_{1 \leq i \leq
    p}\|S_i\|\Bigg\|\sum_{i=1}^pS_i-(n-1)\Sigma_i\Bigg\|\Bigg]\\
    & \leq \frac{1}{p}\Bigg[ E\Big(\max_{i\leq i \leq p}\|S_i\|^2\Big)
    E\Bigg\|\sum_{i=1}^pS_i-(n-1)\Sigma_i\Bigg\|^2\Bigg]^{1/2}\\
    & = O(p^{-1})\times O(qp^{1/4}\log p + qp^{1/4}\log q) \times O(p^{1/2})\\
    & = O(p^{-1/4}\log p) = o(1).
\end{align*}
Hence, we have 
\begin{align}
    \sup_{\nu>q+1, \|\Psi\| \leq \max_{1\leq i \leq p}\|S_i\|}\Bigg|
    p^{-1}\sum_{i=1}^p\frac{2(\nu-q-1)(n-1)}{(\nu+n-q-2)^2} \big\langle \Psi,
    S_i - (n-1)\Sigma_i \big\rangle\Bigg| \stackrel{\text{prob}}{\longrightarrow} 0
    \quad \text{as }p \to \infty.\label{eq:S_1_pf}
\end{align}

For~\eqref{eq:S_2}, we have
\begin{align*}
    & \quad \sup_{\nu>q+1} \Bigg| p^{-1}\sum_{i=1}^p C(\nu)\Big[(\tr
    S_i)^2-(n-1)^2(\tr\Sigma_i)^2 -
    2(n-1)\tr(\Sigma_i^2) \Big]\Bigg|\\
    & = \frac{1}{(n+1)(n-2)(n-1)}\Bigg| p^{-1}\sum_{i=1}^p \Big[(\tr
    S_i)^2-(n-1)^2(\tr\Sigma_i)^2 -
    2(n-1)\tr(\Sigma_i^2) \Big]\Bigg|
\end{align*}
and 
\begin{align*}
    & \quad \var\Bigg[p^{-1}\sum_{i=1}^p \Big[(\tr
    S_i)^2-(n-1)^2(\tr\Sigma_i)^2 - 2(n-1)\tr(\Sigma_i^2) \Big]\Bigg]\\
    & = p^{-2}\sum_{i=1}^p \var\Big((\tr
    S_i)^2-(n-1)^2(\tr\Sigma_i)^2 - 2(n-1)\tr(\Sigma_i^2) \Big)\\
    & = p^{-2}\sum_{i=1}^p E\Big((\tr
    S_i)^2-(n-1)^2(\tr\Sigma_i)^2 - 2(n-1)\tr(\Sigma_i^2) \Big)^2\\
    & = p^{-2}\sum_{i=1}^p \Big[E(\tr
    S_i)^4 - \Big((n-1)^2(\tr\Sigma_i)^2 + 2(n-1)\tr(\Sigma_i^2)\Big)^2\Big]\\
    & = p^{-2}\sum_{i=1}^p
    \Big[2^4\sum_{\kappa}\Big(\frac{n-1}{2}\Big)_\kappa C_\kappa(\Sigma_i)
    - (n-1)^4(\tr\Sigma_i)^4\\
    & \quad - 4(n-1)^3(\tr\Sigma_i)^2\tr(\Sigma_i^2) - 4(n-1)^2
    \big(\tr(\Sigma_i^2)\big)^2 \Big]\stackrel{p \to \infty}{\longrightarrow} 0
    \quad \text{(by~\eqref{eq:trace2})}.
\end{align*}
By Markov's inequality, we have
\begin{align}
    \Bigg| p^{-1}\sum_{i=1}^p \Big[(\tr S_i)^2-(n-1)^2(\tr\Sigma_i)^2 - 2(n-1)\tr(\Sigma_i^2) \Big]\Bigg|\stackrel{\text{prob}}{\longrightarrow} 0 \quad \text{as }p \to \infty.\label{eq:S_2_pf}
\end{align}
\begin{remark}
    By Lemma~\ref{lem:zonal}, assumption (i) implies
    $\limsup_{p\to\infty}p^{-1}\sum_{i=1}^p C_\kappa(\Sigma_i) < \infty$ for all
    partitions $\kappa=(k_1,\ldots,k_q)$ with $\sum_{j=1}^q k_j \leq 4$.  
\end{remark}

For~\eqref{eq:S_3}, we have
\begin{align*}
    & \quad \sup_{\nu>q+1} \Bigg| p^{-1}\sum_{i=1}^p (n-1)C(\nu)\Big[\tr(
        S_i^2)-(n-1)(\tr\Sigma_i)^2 -
    n(n-1)\tr(\Sigma_i^2) \Big]\Bigg|\\
    & = \frac{1}{(n+1)(n-2)}\Bigg| p^{-1}\sum_{i=1}^p\Big[\tr(
        S_i^2)-(n-1)(\tr\Sigma_i)^2 -
    n(n-1)\tr(\Sigma_i^2) \Big]\Bigg|
\end{align*}
and
\begin{align*}
    & \quad \var\Bigg[p^{-1}\sum_{i=1}^p \Big[\tr
    (S_i^2)-(n-1)(\tr\Sigma_i)^2 - n(n-1)\tr(\Sigma_i^2) \Big]\Bigg]\\
    & = p^{-2}\sum_{i=1}^p \var\Big(\tr
    (S_i^2)-(n-1)(\tr\Sigma_i)^2 - n(n-1)\tr(\Sigma_i^2) \Big)\\
    & = p^{-2}\sum_{i=1}^p E\Big(\tr
    (S_i^2)-(n-1)(\tr\Sigma_i)^2 - n(n-1)\tr(\Sigma_i^2) \Big)^2\\
    & = p^{-2}\sum_{i=1}^p \Big[E(\tr
    (S_i^2))^2 - \Big((n-1)(\tr\Sigma_i)^2 + n(n-1)\tr(\Sigma_i^2)\Big)^2\Big]\\
    & \leq p^{-2}\sum_{i=1}^p
    \Big[E(\tr S_i)^4
    - (n-1)^2(\tr\Sigma_i)^4\\
    & \quad - 2n(n-1)^2(\tr\Sigma_i)^2\tr(\Sigma_i^2) - n^2(n-1)^2
    \big(\tr(\Sigma_i^2)\big)^2 \Big]\stackrel{p \to \infty}{\longrightarrow} 0
    \quad \text{(by~\eqref{eq:trace2})}.
\end{align*}
By Markov's inequality, we have
\begin{align}
    \Bigg| p^{-1}\sum_{i=1}^p\Big[\tr( S_i^2)-(n-1)(\tr\Sigma_i)^2 - n(n-1)\tr(\Sigma_i^2) \Big]\Bigg|\stackrel{\text{prob}}{\longrightarrow} 0 \quad \text{as }p \to \infty.\label{eq:S_3_pf}
\end{align}

For~\eqref{eq:S_4}, by Cauchy's inequality, we have
\begin{align*}
    \Bigg| p^{-1}\sum_{i=1}^p 2(n-1)\Big[\tr (S_i\Sigma_i) -
    (n-1)\tr(\Sigma_i^2) \Big]\Bigg|
    & = \Bigg| p^{-1}\sum_{i=1}^p 2(n-1)\langle\Sigma_i, S_i-(n-1)\Sigma_i\rangle\Bigg|\\
    & \leq 2(n-1)p^{-1}\sum_{i=1}^p|\langle\Sigma_i, S_i-(n-1)\Sigma_i\rangle|\\
    & \leq 2(n-1)p^{-1}\sum_{i=1}^p\|\Sigma_i\|\|S_i-(n-1)\Sigma_i\|.
\end{align*}
Then 
\begin{align*}
    & \quad \var\Bigg(2(n-1)p^{-1}\sum_{i=1}^p\|\Sigma_i\|\|S_i-(n-1)\Sigma_i\|\Bigg)\\
    & = \frac{4(n-1)^2}{p}p^{-1}\sum_{i=1}^p
    \var\big(\|\Sigma_i\|\|S_i-(n-1)\Sigma_i\|\big)\\ 
    & \leq \frac{4(n-1)^2}{p}p^{-1}\sum_{i=1}^p
    E\big(\|\Sigma_i\|^2\|S_i-(n-1)\Sigma_i\|^2\big)\\ 
    & = \frac{4(n-1)^2}{p}p^{-1}\sum_{i=1}^p
    \tr(\Sigma_i^2)\Big[E\tr(S_i^2) - 2(n-1)E\tr(S_i\Sigma_i) +
    (n-1)^2\tr(\Sigma_i^2)\Big]\\ 
    & = \frac{4(n-1)^2}{p}p^{-1}\sum_{i=1}^p
    \tr(\Sigma_i^2)\Big[n(n-1)\tr(\Sigma_i^2) + (n-1)(\tr\Sigma_i)^2 -
        2(n-1)^2\tr(\Sigma_i^2) + (n-1)^2\tr(\Sigma_i^2)\Big]\\ 
    & = \frac{4(n-1)^3}{p}p^{-1}\sum_{i=1}^p\big(\tr(\Sigma_i^2)\big)^2 +
    \tr(\Sigma_i^2)(\tr\Sigma_i)^2 \stackrel{p \to \infty}{\longrightarrow} 0.
\end{align*}
By Markov's inequality, we have
\begin{align}
    \Bigg| p^{-1}\sum_{i=1}^p 2(n-1)\Big[\tr (S_i\Sigma_i) -
    (n-1)\tr(\Sigma_i^2) \Big]\Bigg|\stackrel{\text{prob}}{\longrightarrow} 0
    \quad \text{as }p \to \infty.\label{eq:S_4_pf}
\end{align}

For~\eqref{eq:S_5}, we have 
\begin{align*}
    & \quad \sup_{\nu>q+1, \|\Psi\|\leq \max_{1\leq i \leq p}\|S_i\|}\Bigg|
    p^{-1}\sum_{i=1}^p\frac{2\big\langle \Psi-(\nu-q-1)\Sigma_i, S_i - (n-1)\Sigma_i
    \big\rangle}{(\nu+n-q-2)^2}\Bigg|\\
    & \leq \sup_{\nu>q+1, \|\Psi\|\leq \max_{1\leq i \leq p}\|S_i\|}\Bigg|
    p^{-1}\sum_{i=1}^p\frac{2\big\langle \Psi, S_i - (n-1)\Sigma_i
    \big\rangle}{(\nu+n-q-2)^2}\Bigg|\\
    & \qquad + \sup_{\nu>q+1}\Bigg|
    p^{-1}\sum_{i=1}^p\frac{2(\nu-q-1)\big\langle \Sigma_i, S_i - (n-1)\Sigma_i
    \big\rangle}{(\nu+n-q-2)^2}\Bigg|\\
    & \leq \sup_{\|\Psi\|\leq \max_{1\leq i \leq p}\|S_i\|}\Bigg|
    p^{-1}\sum_{i=1}^p2\big\langle \Psi, S_i - (n-1)\Sigma_i
    \big\rangle\Bigg| + \Bigg|
    p^{-1}\sum_{i=1}^p2\big\langle \Sigma_i, S_i - (n-1)\Sigma_i
    \big\rangle\Bigg|.
\end{align*}
Since by assumption (i)
\begin{align*}
    & \quad\var\Bigg(p^{-1}\sum_{i=1}^p2\big\langle \Sigma_i, S_i - (n-1)\Sigma_i
    \big\rangle\Bigg)\\
    & = \frac{1}{p^2}\sum_{i=1}^p E\big[\langle \Sigma_i,
    S_i-(n-1)\Sigma_i\rangle^2\big]\\ 
    & \leq \frac{1}{p^2}\sum_{i=1}^p E\big[\|\Sigma_i\|^2
    \|S_i-(n-1)\Sigma_i\|^2\big]\\ 
    & = \frac{1}{p^2}\sum_{i=1}^p \|\Sigma_i\|^2
    E\big[\tr(S_i^2)-2(n-1)\tr(S_i\Sigma_i) + (n-1)^2\tr(\Sigma_i^2)\big]\\ 
    & = \frac{1}{p^2}\sum_{i=1}^p \tr(\Sigma_i^2)
    \Big[(n-1)\tr(\Sigma_i^2) + (n-1)(\tr\Sigma_i)^2\Big]\\ 
    & \leq \frac{n-1}{p^2}\sum_{i=1}^p(\tr\Sigma_i)^4
    \stackrel{p\to\infty}{\longrightarrow} 0,
\end{align*}
by Markov's inequality, we have 
\begin{align*}
    \Bigg|p^{-1}\sum_{i=1}^p2\big\langle \Sigma_i, S_i -
    (n-1)\Sigma_i\big\rangle\Bigg| \stackrel{\text{prob}}{\longrightarrow} 0
    \quad \text{as }p \to \infty.
\end{align*}
Similarly, we have 
\begin{align*}
    \sup_{\|\Psi\|\leq \max_{1\leq i \leq p}\|S_i\|}\Bigg|
    p^{-1}\sum_{i=1}^p2 \big\langle \Psi,
    S_i - (n-1)\Sigma_i \big\rangle\Bigg| & \leq  
    \frac{2}{p}\max_{1\leq i \leq
    p}\|S_i\|\Bigg\|\sum_{i=1}^pS_i-(n-1)\Sigma_i\Bigg\|
\end{align*}
and
\begin{align*}
    E\Bigg[ \frac{2}{p}\max_{1\leq i \leq
    p}\|S_i\|\Bigg\|\sum_{i=1}^pS_i-(n-1)\Sigma_i\Bigg\|\Bigg] & = o(1).
\end{align*}
Hence, we have 
\begin{align}
    \sup_{\nu>q+1, \|\Psi\|\leq \max_{1\leq i \leq p}\|S_i\|}\Bigg|
    p^{-1}\sum_{i=1}^p\frac{2\big\langle \Psi-(\nu-q-1)\Sigma_i, S_i - (n-1)\Sigma_i
    \big\rangle}{(\nu+n-q-2)^2}\Bigg|
     \stackrel{\text{prob}}{\longrightarrow} 0
    \quad \text{as }p \to \infty.\label{eq:S_5_pf}
\end{align}

Combining \eqref{eq:S_1_pf}, \eqref{eq:S_2_pf}, \eqref{eq:S_3_pf},
\eqref{eq:S_4_pf}, and \eqref{eq:S_5_pf}, we have 
\begin{align}
    \sup_{\nu>q+1, \|\Psi\|\leq \max_{1\leq i \leq p}\|S_i
    \|}\big|\text{SURE}_2(\nu,\Psi) - L_2\big(\widehat{\bd{\Sigma}}^{\nu,\Psi},
    \bd{\Sigma}\big)\big|  \stackrel{\text{prob}}{\longrightarrow} 0 \quad \text{as
    }p\to\infty. \label{eq:S_pf}
\end{align}
The proof is concluded by \eqref{eq:M_pf} and \eqref{eq:S_pf}.
\end{proof}

\begin{theorem}\label{thm4}
    If assumptions (i), (ii), and (iii) in Theorem~\ref{thm3} hold, then
    \begin{align*}
        \lim_{p\to\infty} \Bigg[ R\Big(\big( \widehat{\bd{M}}^{\text{SURE}},
        \widehat{\bd{\Sigma}}^{\text{SURE}} \big), \big( \bd{M},
        \bd{\Sigma} \big)\Big) - R\Big(\big(
        \widehat{\bd{M}}^{\lambda,\mu}, \widehat{\bd{\Sigma}}^{\Psi,
        \nu} \big), \big( \bd{M}, \bd{\Sigma} \big)\Big)\Bigg] \leq 0.       
    \end{align*}
\end{theorem}
\begin{proof}
    Since
    \begin{align*}
        & \quad L\Big(\big(\widehat{\bd{M}}^{\text{SURE}},
        \widehat{\bd{\Sigma}}^{\text{SURE}}\big), \big(\bd{M},
        \bd{\Sigma}\big)\Big) - L\Big(\big(\widehat{\bd{M}}^{\lambda,\mu},
        \widehat{\bd{\Sigma}}^{\Psi, \nu}\big), \big(\bd{M},
        \bd{\Sigma}\big)\Big)\\
        & = L\Big(\big(\widehat{\bd{M}}^{\text{SURE}},
        \widehat{\bd{\Sigma}}^{\text{SURE}}\big), \big(\bd{M},
        \bd{\Sigma}\big)\Big) - \text{SURE}(\widehat{\lambda}^{\text{SURE}},
        \widehat{\mu}^{\text{SURE}}, \widehat{\Psi}^{\text{SURE}},
        \widehat{\nu}^{\text{SURE}})\\
        & \qquad + \text{SURE}(\widehat{\lambda}^{\text{SURE}},
        \widehat{\mu}^{\text{SURE}}, \widehat{\Psi}^{\text{SURE}},
        \widehat{\nu}^{\text{SURE}}) - \text{SURE}(\lambda, \mu, \Psi,\nu)\\
        & \qquad + \text{SURE}(\lambda, \mu, \Psi,\nu) -
        L\Big(\big(\widehat{\bd{M}}^{\lambda,\mu}, \widehat{\bd{\Sigma}}^{\Psi,
        \nu}\big), \big(\bd{M}, \bd{\Sigma}\big)\Big)\\
        & \leq \sup_{\lambda,\mu,\Psi,\nu} \Big|
        L\Big(\big(\widehat{\bd{M}}^{\lambda,\mu}, \widehat{\bd{\Sigma}}^{\Psi,
        \nu}\big), \big(\bd{M}, \bd{\Sigma}\big)\Big) -
        \text{SURE}(\lambda,\mu,\Psi,\nu)\Big|\\
        & \qquad + 0\\
        & \qquad + \sup_{\lambda,\mu,\Psi,\nu} \Big|
        L\Big(\big(\widehat{\bd{M}}^{\lambda,\mu}, \widehat{\bd{\Sigma}}^{\Psi,
        \nu}\big), \big(\bd{M}, \bd{\Sigma}\big)\Big) -
        \text{SURE}(\lambda,\mu,\Psi,\nu)\Big|\\
        & = 2 \sup_{\lambda,\mu,\Psi,\nu} \Big|
        L\Big(\big(\widehat{\bd{M}}^{\lambda,\mu}, \widehat{\bd{\Sigma}}^{\Psi,
        \nu}\big), \big(\bd{M}, \bd{\Sigma}\big)\Big) -
        \text{SURE}(\lambda,\mu,\Psi,\nu)\Big|,
    \end{align*}
    from Theorem~\ref{thm3}, we have
    \[
        \lim_{p \to \infty} \Big[L\Big(\big(\widehat{\bd{M}}^{\text{SURE}},
        \widehat{\bd{\Sigma}}^{\text{SURE}}\big), \big(\bd{M},
        \bd{\Sigma}\big)\Big) - L\Big(\big(\widehat{\bd{M}}^{\lambda,\mu},
        \widehat{\bd{\Sigma}}^{\Psi, \nu}\big), \big(\bd{M},
        \bd{\Sigma}\big)\Big)\Big] \leq 0. 
    \]
    Hence, by dominated convergence, we have
    \begin{align*}
        & \quad \lim_{p\to\infty} \Big[R\Big(\big(
        \widehat{\bd{M}}^{\text{SURE}},
        \widehat{\bd{\Sigma}}^{\text{SURE}} \big), \big( \bd{M},
        \bd{\Sigma} \big)\Big) - R\Big(\big(
        \widehat{\bd{M}}^{\lambda,\mu}, \widehat{\bd{\Sigma}}^{\Psi,
        \nu} \big), \big( \bd{M}, \bd{\Sigma} \big)\Big)\Big]\\
        & = \lim_{p \to \infty}E\Big[L\Big(\big(\widehat{\bd{M}}^{\text{SURE}},
        \widehat{\bd{\Sigma}}^{\text{SURE}}\big), \big(\bd{M},
        \bd{\Sigma}\big)\Big) - L\Big(\big(\widehat{\bd{M}}^{\lambda,\mu},
        \widehat{\bd{\Sigma}}^{\Psi, \nu}\big), \big(\bd{M},
        \bd{\Sigma}\big)\Big)\Big]\\
        & = E\Bigg\{\lim_{p \to \infty} \Big[L\Big(\big(\widehat{\bd{M}}^{\text{SURE}},
        \widehat{\bd{\Sigma}}^{\text{SURE}}\big), \big(\bd{M},
        \bd{\Sigma}\big)\Big) - L\Big(\big(\widehat{\bd{M}}^{\lambda,\mu},
        \widehat{\bd{\Sigma}}^{\Psi, \nu}\big), \big(\bd{M},
        \bd{\Sigma}\big)\Big)\Big]\Bigg\}\\
        & \leq 0.       
    \end{align*}
\end{proof}
    
\section{Implementation Details}
Note that to find the shrinkage estimators $(\widehat{\bd{M}}^{\lambda,\mu}, \widehat{\bd{\Sigma}}^{\Psi, \nu})$, we need to solve the optimization problem
\[
\min_{\lambda, \mu, \Psi, \nu} \text{SURE}(\lambda, \mu, \Psi, \nu)
\]
which is a non-convex optimization problem. Hence the solution depends heavily on the initialization of the minimization algorithm. In this section, we provide a way to choose the initialization so that, in our experiments, the algorithm converges successfully in less than 10 iterations. We can compute the marginal expectations
\begin{align}
    E^{\text{LE}}\big(\bar{X}_i^{\text{LE}}\big) & = E_{M_i}^{\text{LE}}\Big[E^{\text{LE}}_{X}\big(\bar{X}_i^{\text{LE}}|M_i\big)\Big] = E_{M_i}^{\text{LE}}[M_i] = \mu \label{eq:marginal_m}\\
    Ed^2_{\text{LE}}(\bar{X}_i^{\text{LE}}, \mu) & = E_{\Sigma_i}\big\{E_{M_i}\big[E_{X}\big(d^2_{\text{LE}}(\bar{X}_i^{\text{LE}}, \mu)|M_i,\Sigma_i\big)|\Sigma_i\big]\big\} \nonumber \\
    & = E_{\Sigma_i}\Big[\Big(1/n+1/\lambda\Big)\tr\Sigma_i\Big] = \Bigg(\frac{1}{n} + \frac{1}{\lambda}\Bigg)\frac{\tr\Psi}{\nu-q-1} \label{eq:marginal_d2}\\
    E(S_i) & = E_{\Sigma_i}[E_{S_i}(S_i|\Sigma_i)] = E_{\Sigma_i}[(n-1)\Sigma_i] = \frac{n-1}{\nu-q-1}\Psi\label{eq:marginal_S}\\
    E(S_i^{-1}) & = E_{\Sigma_i}[E_{S_i}(S_i^{-1}|\Sigma_i)] = E_{\Sigma_i}\Big[\frac{\Sigma_i^{-1}}{n-q-2}\Big] = \frac{\nu}{n-q-2}\Psi^{-1}\label{eq:marginal_Sinv}
\end{align}
where $E^{\text{LE}}(\cdot)$ denotes the Fr\'{e}chet expectation with respect to the Log-Euclidean metric, i.e.\ $E^{\text{LE}}(X) = \exp(E(\log X))$. Thus the hyperparameters can be written as
\begin{align*}
    \mu & = E^{\text{LE}}\big(\bar{X}_i^{\text{LE}}\big)\\
    \lambda & = \frac{nEd^2_{\text{LE}}(\bar{X}_i^{\text{LE}}, \mu)}{\frac{n}{n-1}E(\tr S_i) - Ed^2_{\text{LE}}(\bar{X}_i^{\text{LE}}, \mu)}\quad \text{(by \eqref{eq:marginal_d2} and \eqref{eq:marginal_S})}\\
    \nu & = \frac{q+1}{\frac{n-q-2}{q(n-1)}\tr(E(S_i)E(S_i^{-1})) - 1} + q + 1 \quad \text{(by \eqref{eq:marginal_S} and \eqref{eq:marginal_Sinv})}\\
    \Psi & = \frac{\nu-q-1}{n-1}E(S_i)\quad \text{(by \eqref{eq:marginal_S})}
\end{align*}
and an initialization of the hyperparameters can be obtained by replacing the (Fr\'{e}chet) expectations with the corresponding sample (Fr\'{e}chet) means, i.e.\ 
\begin{align*}
    \mu_0 & = \exp\Bigg(p^{-1}\sum_{i=1}^p\log\bar{X}_i^{\text{LE}}\Bigg)\\
    \lambda_0 & = \frac{np^{-1}\sum_{i=1}^pd^2_{\text{LE}}(\bar{X}_i^{\text{LE}}, \mu_0)}{\frac{n}{p(n-1)}\sum_{i=1}^p\tr S_i - p^{-1}\sum_{i=1}^pd^2_{\text{LE}}(\bar{X}_i^{\text{LE}}, \mu_0)}\\
    \nu_0 & = \frac{q+1}{\frac{n-q-2}{p^2q(n-1)}\tr\Big[(\sum_{i=1}^p S_i)(\sum_{i=1}^pS_i^{-1})\Big] - 1} + q + 1\\
    \Psi_0 & = \frac{\nu_0-q-1}{p(n-1)}\sum_{i=1}^pS_i.
\end{align*}
Note that these initial values can also be viewed as empirical Bayes estimates for $\mu$, $\lambda$, $\nu$, and $\Psi$ obtained by matching moments. However these estimates do not possess the asymptotic optimality as stated in Theorem~\ref{thm3} and Theorem~\ref{thm4} since they are not obtained by minimizing an estimate of the risk function.

\bibliographystyle{agsm}

\bibliography{reference}

\end{document}